\documentclass[12pt]{elsarticle}

\usepackage{mathstyle00}

\usepackage{verbatim} 

\usepackage[margin=0.92in]{geometry}    

\begin{document}
\begin{frontmatter}
\title{Convergence analysis of a Lagrangian numerical scheme in computing effective diffusivity of 3D time-dependent flows}
		
       \author[uoc]{Zhongjian Wang}
        \ead{zhongjian@statistics.uchicago.edu}
		\author[uci]{Jack Xin}
		\ead{jxin@math.uci.edu}
		\author[hku]{Zhiwen Zhang\corref{cor1}}
		\ead{zhangzw@hku.hk}
		
		\address[uoc]{Department of Statistics, The University of Chicago, Chicago, IL 60637, USA.}
		\address[uci]{Department of Mathematics, University of California at Irvine, Irvine, CA 92697, USA.}
		\address[hku]{Department of Mathematics, The University of Hong Kong, Pokfulam Road, Hong Kong SAR, China.}		
		\cortext[cor1]{Corresponding author}
\begin{abstract}
In this paper, we {\color{black} study the convergence analysis for a robust stochastic structure-preserving Lagrangian numerical scheme in computing effective diffusivity of time-dependent chaotic flows, which are modeled by stochastic differential equations (SDEs).} Our numerical scheme is based on a splitting method to solve the corresponding SDEs in which the deterministic subproblem is discretized using structure-preserving schemes while the random subproblem is discretized using the Euler-Maruyama scheme. 
We obtain a sharp and uniform-in-time convergence analysis for the proposed numerical scheme that allows us to accurately compute long-time solutions of the SDEs. As such, we can compute the effective diffusivity for time-dependent flows. Finally, we present numerical results to demonstrate the accuracy and efficiency of the proposed method in computing effective diffusivity for the time-dependent Arnold-Beltrami-Childress (ABC) flow and Kolmogorov flow in three-dimensional space.   
\end{abstract}
		
\end{frontmatter}

\section{Introduction} \label{sec:Introduction}
\noindent
In this paper, we study the convection-enhanced diffusion phenomenon for particles moving in time-dependent chaotic flows, which is defined by the following passive tracer model, i.e., a stochastic differential equation (SDE),
\begin{align}\label{eqn:generalSDED-TDFlow}
\mathrm{d}\textbf{X}(t) = \textbf{v}(t,\textbf{X})\mathrm{d}t + \sigma \mathrm{d}\textbf{w}(t),  \quad  \textbf{X}\in \mathbb{R}^{d},   
\end{align}
where $\textbf{X}(t)=(x_1(t),...,x_d(t))^{T} \in \mathbb{R}^{d}$ is the position of the particle, $\sigma>0$ is the molecular diffusivity, and $\{{\bf w}(t)\}_{t\ge0}$ is the standard $d$-dimensional Browinian motion. The velocity field $\textbf{v}(t,\textbf{x})$ is time-dependent and divergence free, i.e. $\nabla_{\textbf{x}}\cdot \textbf{v}(t,\textbf{x}) = 0$, for all $t\geq 0$. In order to guarantee the existence of the solution to the SDE \eqref{eqn:generalSDED-TDFlow}, we also assume that $\textbf{v}(t,\textbf{x})$ is Lipschitz in $\textbf{x}$. The passive tracer model \eqref{eqn:generalSDED-TDFlow} has many applications in physical and engineering sciences, including atmosphere science, ocean science, chemical engineering, and combustion.

We will study the long-time large-scale behavior of the particle $\textbf{X}(t)$ in the passive tracer model \eqref{eqn:generalSDED-TDFlow}, i.e., whether the motion of the particle  ${\bf X}(t)$ has a long-time diffusive limit. Let $\textbf{X}^{\epsilon}(t)\equiv\epsilon \textbf{X}(t/\epsilon^2)$ denote the rescaled process of \eqref{eqn:generalSDED-TDFlow}. 
We aim to investigate whether $\textbf{X}^{\epsilon}(t)$ converge in law to a Brownian motion with a covariance matrix $D^{E}\in \mathbb{R}^{d\times d}$ as $\epsilon\rightarrow0$, where $D^{E}$ is called the effective diffusivity matrix. The dependence of $D^E$ on the velocity field of the passive tracer model is complicated and highly nontrivial. There are many theoretical works, where the homogenization theory was applied to study the effective diffusivity matrix $D^E$ of the passive tracer model with spatial periodic velocity fields or random velocity fields with short-range correlations; see e.g. \cite{BensoussanLionsPapa:2011,Garnier:97,Oleinik:94,Stuart:08} and references therein. 

For many complicated velocity fields of physical interests, one cannot apply the homogenization theory to {\color{black}compute} the corresponding effective diffusivity matrix $D^E$, or even determine its existence. Therefore, many numerical 
methods were developed to compute $D^E$. These results include, among others, for time-independent Taylor-Green flows, the authors of \cite{StuartZygalakis:09} proposed a stochastic splitting method and calculated the effective diffusivity in the limit of vanishing molecular diffusion. For time-dependent chaotic flows, 
an efficient model reduction method based on the spectral method was developed to compute $D^E$ using the Eulerian framework \cite{JackXinLyu:2017}. The reader is referred to \cite{Majda:99} for an extensive review of many existing mathematical theories and numerical simulations for the passive tracer model with different velocity fields.  

Recently, we developed a robust structure-preserving Lagrangian scheme to compute the effective diffusivity for chaotic and stochastic flows in \cite{WangXinZhang:18}. 
We also obtained a rigorous error estimate for the numerical scheme in \cite{WangXinZhang:18}. Specifically, let $D^{E,num}$ denote the numerical effective diffusivity obtained by our method. We got the error estimate, $||D^{E,num}-D^E||\le C\Delta t+C(T)(\Delta t)^2$, where the computational time $T$ should be greater than the diffusion time (also known as mixing time). This error estimate is not sharp in the sense that the pre-factor $C(T)$ may grow fast with respect to $T$, since the error estimation is based on a Gronwall inequality technique. 
Later, we obtained a sharp convergence rate for our numerical scheme and got rid of the term $C(T)$ in the error estimate in \cite{Zhongjian2018sharp}. However, this result only holds for the passive tracer models with time-independent flows.

In this paper, we aim to obtain a sharp convergence analysis for our numerical scheme in computing the effective diffusivity of passive tracer models in spatial-temporal periodic velocity fields. These types of flow fields 
are well-known for exhibiting chaotic streamlines and have many applications in turbulent diffusion \cite{Majda:99}. Since in this case the velocity field depends on the temporal variable, the generator associated with the stochastic process, i.e. the solution  $\textbf{X}(t)$ in Eq.\eqref{eqn:generalSDED-TDFlow} becomes {\color{black} non-autonomous. The generator is now a parabolic-type operator (see Eq.\eqref{DefinitionGenerator}), instead of an elliptic-type operator that was studied in \cite{Zhongjian2018sharp} when the flows are time-independent. The uniqueness only happens regarding each  time period. Hence the extension of the analysis developed in \cite{Zhongjian2018sharp} to time-dependent flows is not straightforward. We will develop new techniques to overcome this difficulty; see Theorem \ref{thm:main-per} and Lemma \ref{lem:existenceofcellpro} in Section \ref{sec:ConvergenceAnalysis}. We also emphasize that when the flows are time-independent, we can construct the ballistic orbits of the ABC flows and Kolmogorov flows and study their dynamic behaviors. When the flows are time-dependent, however, their streamlines are complicated, which makes it is difficult to construct the orbits of these flows.} 

Though there are several prior works on structure-preserving schemes for ODEs and SDEs, see e.g.{\color{black} \cite{ErnstLubich:06,hong2006multi,tao2010nonintrusive,JanHesthaven2017structure,lelievre_stoltz_2016} and references therein, our work has several novel contributions. The first novelty is the convergence analysis, where we develop new techniques to deal with time-dependent flows. 
To handle the parabolic-type generator, we pile up snapshots of each time step within a single time period together. By viewing the numerical solutions as a Markov process and exploring the ergodicity of the solution process, we succeed in obtaining a sharp convergence analysis for our method in computing the effective diffusivity, where the error estimate does not depend on the computational time. Therefore, we can compute the long-time solutions of passive tracer models without losing accuracy; see Fig.\ref{fig:tdkflowVaryD0VaryEps} and Fig.\ref{fig:tdABCflowVaryD0VaryEps}. If we choose  the Gronwall inequality in the error estimate, we cannot get rid of the exponential growth pre-factor in the error term, which makes the convergence analysis not sharp. Most importantly, our convergence result reveals the equivalence of the definition of effective diffusivity using the Eulerian framework and the Lagrangian framework; see Theorem \ref{thm:convergence}, which is fundamental and important. For 3D time-dependent flow problems, the Eulerian framework has good theoretical value but the Lagrangian framework is computationally accessible.

Another novelty is that the stochastic structure-preserving Lagrangian scheme is robust and quite cheap in computing the long-time solutions of the passive tracer model \eqref{eqn:generalSDED-TDFlow}, especially for problems in three-dimensional space. If one adopts the Eulerian framework to compute the effective diffusivity of the passive tracer model \eqref{eqn:generalSDED-TDFlow}, one needs to solve a convection-diffusion-type cell problem; see Eq.\eqref{CellProblem_EffectiveDiffusivity}. When the molecular diffusivity $\sigma$ is small and/or the dimension of spatial variables is big, say $d=3$, it is exorbitantly expensive to solve the cell problems. As indicated in Eq.\eqref{CellProblem_EffectiveDiffusivity2}, the effective diffusivities depend on the integration of the gradient of the solution to the cell problem. In many cases, e.g. time-dependent ABC flow, the effective diffusivities grows up rapidly as $\sigma$ decreases; see Fig.\ref{fig:tdabcflowresult}. In our Lagrangian approach, we can overcome the difficulties of long-time integration of the SDEs (raised as $\sigma$ decreases) by using robust structure-preserving schemes. However, for the Eulerian approach, one needs to solve a four-dimensional PDE (three variables in space dimension and one variable in the time dimension) and solutions have sharp gradients as the diffusivity decreases, which makes the Eulerian approach for computing effective diffusivities expensive.}

Numerical results show that our Lagrangian scheme is insensitive to the molecular diffusivity $\sigma$ and computational cost linearly depends on the dimension $d$ of spatial variables in the passive tracer models \eqref{eqn:generalSDED-TDFlow}. Thus, we are able to investigate the convection-enhanced diffusion phenomenon for several typical time-dependent chaotic flows of physical interests, including the time-dependent ABC flow and the time-dependent Kolmogorov flow in three-dimensional space. We discover that the maximal enhancement is achieved in the former case, while a submaximal enhancement is observed in the latter case; see Fig.\ref{fig:tdabcflowresult} and Fig.\ref{fig:tdkflowVaryD0OneEps}, respectively. In addition, we find that the level of chaos and the strength of diffusion enhancement seem to compete with each other in the time-dependent ABC flow; see Fig.\ref{fig:tdABCflowVaryFreq}. To the best of our knowledge, our work appears to be the first one in the literature to develop numerical methods to study the convection-enhanced diffusion phenomenon in 3D time-dependent flows. 
 
The rest of the paper is organized as follows. In Section 2, we give the definition of the effective diffusivity matrix using the Eulerian framework and the Lagrangian framework. In Section 3, we propose the stochastic structure-preserving Lagrangian scheme in computing effective diffusivity for the passive tracer model \eqref{eqn:generalSDED-TDFlow}. In Section 4, we provide a sharp convergence analysis for the proposed method based on a probabilistic approach. In addition, we shall show that our method can be used to solve high-dimensional flow problems and the error estimate can be obtained naturally. In Section 5, we present numerical results to demonstrate the accuracy and efficiency of our method. We also investigate the convection-enhanced diffusion phenomenon for time-dependent chaotic flows. Concluding remarks are made in Section 6.

\section{Effective diffusivity  of the passive tracer models} \label{sec:EffectiveDiffusivity} 
%
\noindent
There are two frameworks to define the effective diffusivity of the passive tracer models. We first 
discuss the Eulerian framework. One natural way to study the expectation of the paths for the SDE given by the Eq.\eqref{eqn:generalSDED-TDFlow} is to consider its associated backward Kolmogorov equation \cite{Oksendal:13}. 
Due to the time-dependence nature of the velocity field, we need to deal with a space-time ergodic random flow. Specifically, given  a sufficiently smooth function $\phi(\tau,\textbf{x})$ in $\mathbb{R}\times\mathbb{R}^{d}$, let $u(t,\tau,\textbf{x})=\mathbb{E}\big[\phi(t+\tau,\textbf{X}_{t+\tau})|\textbf{X}_{\tau}=\textbf{x}\big]$ and $\textbf{X}(t)$ be the solution to Eq.\eqref{eqn:generalSDED-TDFlow}. Then, $u(t,\tau,\textbf{x})$ satisfies
the following backward Kolmogorov equation  
\begin{align}\label{BackwardKolmolgorovEquation0}
u_{t}=\mathcal{L}u, \quad u(0,\tau,\textbf{x})=\phi(\tau,\textbf{x}).
\end{align} 
In Eq.\eqref{BackwardKolmolgorovEquation0}, the generator $\mathcal{L}$ is defined as 
\begin{align}\label{DefinitionGenerator}
\mathcal{L}u = \partial_\tau{\color{black}u} + \textbf{v}\cdot \nabla_{\textbf{x}} u + D_{0}\Delta_{\textbf{x}} u,
\end{align}
where $D_0=\sigma^2/2$ is the diffusion coefficient, $\textbf{v}$ is the velocity field, and $\nabla_{\textbf{x}}$ 
and $\Delta_{\textbf{x}}$ denote the gradient operator and Laplace operator, respectively. 
\begin{remark}
Let $\rho(t,\tau,\textbf{x})$ denote the density function of the particle $(t+\tau,\textbf{X}(t+\tau))$ of Eq.\eqref{eqn:generalSDED-TDFlow}. One can define the adjoint operator $\mathcal{L}^{*}$ as $\mathcal{L}^{*}\rho =-\partial_\tau{\color{black}\rho}-\nabla\cdot(\textbf{v} \rho)+D_{0}\Delta\rho$.  Then,  $\rho$	satisfies the Fokker-Planck equation 	$\rho_{t}=\mathcal{L}^{*} \rho$ with the initial density $\rho(t,\tau,\textbf{x})=\rho_0(\tau,\textbf{x})$.
\end{remark}

When $\textbf{v}$ is incompressible (i.e. $\nabla_{\textbf{x}}\cdot \textbf{v}(t,\textbf{x})=0,\ \forall t$), deterministic and space-time periodic in $O(1)$ scale, where we assume the period of $\textbf{v}$ is $1$ in each dimension of the physical and temporal space, the formula for the effective diffusivity matrix is \cite{BensoussanLionsPapa:2011,Stuart:08}
\begin{align}
D^{E} = D_0I +\big\langle \textbf{v}(t,\textbf{x})\otimes \boldsymbol\chi(t,\textbf{x}) \big\rangle_{p},
\label{Def_EffectiveDiffusivity_Euler}
\end{align}
where we have assumed that the fluid velocity $\textbf{v}(t,\textbf{x})$ is smooth and 
the (vector) corrector field $\boldsymbol\chi$ satisfies the cell problem,
\begin{align}
\mathcal{L}\boldsymbol\chi = - \textbf{v}(t,\textbf{y}), \quad (t,\textbf{y})\in \mathbb{T}\times\mathbb{T}^d,
\label{CellProblem_EffectiveDiffusivity}
\end{align}
and $\langle \cdot \rangle_{p} $ denotes temporal and spatial average over $\mathbb{T}\times\mathbb{T}^d$. 
Since $\textbf{v}$ is incompressible, the solution $\boldsymbol\chi$ to the cell problem 
\eqref{CellProblem_EffectiveDiffusivity} is unique up to an additive constant by the Fredholm alternative. By multiplying $\boldsymbol\chi$ to Eq.\eqref{CellProblem_EffectiveDiffusivity}, integrating the corresponding result in $\mathbb{T}\times\mathbb{T}^d$ and using the periodic conditions of $\boldsymbol\chi$ and $\textbf{v}$, we get an equivalent formula for the effective diffusivity as follows
\begin{align}
D^{E} = D_0I + \big\langle \nabla \boldsymbol\chi(t,\textbf{x})\otimes \nabla\boldsymbol\chi(t,\textbf{x}) \big\rangle_{p}.
\label{CellProblem_EffectiveDiffusivity2}
\end{align}
The correction to $D_0I$ in Eq.\eqref{CellProblem_EffectiveDiffusivity2} is nonnegative definite. We can see that $\textbf{e}^{T}D^E\textbf{e}\geq D_0 $ for all unit column vectors $\textbf{e}\in \mathbb{R}^{d}$, which is called  convection-enhanced diffusion. By using a variational principle for time-periodic velocity flows, one can find a upper bound for the effective diffusivity, i.e., 
there exists a nonzero unit column vector $\textbf{e}\in \mathbb{R}^{d}$, such that  
\begin{equation}\label{eqn:maximaldiffusion}
\textbf{e}^{T}D^E\textbf{e} \sim \frac{1}{D_0}, \quad \text{as } D_0\to 0,
\end{equation} 
which is known as the maximal enhancement. More details of the derivation can be found in \cite{Biferale:95,mezic1996maximal,Fannjiang:94}. 
We point out that many theoretical results were built upon the passive tracer models with time-independent 
flows. We are interested in studying the convection-enhanced diffusion phenomenon for time-dependent chaotic flows in this paper. Especially, whether the time-dependent chaotic flows still have the maximal enhancement.


In practice, the cell problem \eqref{CellProblem_EffectiveDiffusivity} can be solved using numerical methods,
such as finite element methods and spectral methods. 
However, when $D_0$ becomes extremely small, the solutions of the cell problem \eqref{CellProblem_EffectiveDiffusivity} develop sharp gradients and demand a large number of finite element basis or Fourier basis to resolve, which makes the Eulerian framework expensive. In addition, when the dimension of spatial variables is big, say $d=3$, the Eulerian framework becomes expensive too. 

Alternatively, one can use the Lagrangian framework to compute the effective diffusivity matrix, which is defined as follows, 
\begin{align}
D_{ij}^{E}=\lim_{t\rightarrow\infty}\frac{\Big\langle\big(x_i(t)-x_i(0))(x_j(t)-x_j(0)\big)\Big\rangle}{2t},
\quad 1\leq i,j \leq d,
\label{Def_EffectiveDiffusivity_Lagrangian}
\end{align}
where $\textbf{X}(t)=(x_1(t),...,x_d(t))^{T}$ is the position of a particle tracer at time $t$ and the average $\langle\cdot\rangle $ is taken over an ensemble of particles.
If the above limit exists, that means the transport of the particle is a standard diffusion process, at least on a long-time scale. 
For example, when the velocity field is the Taylor-Green velocity field \cite{Fannjiang:94,StuartZygalakis:09}, the long-time and large-scale behavior of the passive tracer model is a diffusion process. However, there are cases showing that the spreading of particles does not grow linearly with time but has a power-law $t^{\gamma}$, where $\gamma>1$ and $\gamma<1$ correspond to super-diffusive and sub-diffusive behaviors, respectively; see e.g. \cite{Biferale:95,Majda:99,BenOwhadi2003}.

We shall use the Lagrangian framework in this paper. The Lagrangian framework has the advantages that: (1) it is easy to implement; (2) it does not directly suffer from a small molecular diffusion coefficient $\sigma$ during the computation; and (3) its computational cost only linearly depends on the dimension of spatial variables in the passive tracer models. However, the major difficulty in solving Eq.\eqref{eqn:generalSDED-TDFlow} 
is that the computational time should be long enough to approach the diffusion time scale. To address this challenge, we shall develop robust numerical schemes, which are structure-preserving and accurate for long-time integration. Moreover, we aim to develop the convergence analysis of the proposed numerical schemes. Finally, we shall investigate the relationship between parameters of the time-dependent chaotic  flows and the corresponding effective diffusivity.  

\section{Stochastic structure-preserving schemes}\label{sec:NewStochasticIntegrators}
\subsection{Derivation of numerical schemes}\label{sec:DerivationSchemes}
\noindent
To demonstrate the main idea, we first construct stochastic structure-preserving schemes for a two-dimensional passive tracer model. The derivation of the numerical schemes for high-dimensional passive tracer models will be discussed in Section \ref{sec:HighDimensionalCases}. Specifically, let $\textbf{X}=(x_1,x_2)^{T}$ denote the position of the particle, then the model can be written as   
\begin{equation}\label{eqn:particleSDE}
\begin{cases}
\mathrm{d}x_1=v_1\mathrm{d}t+\sigma \mathrm{d}w_{1,t}, \quad  x_1(0)=x_{1,0}, \\
\mathrm{d}x_2=v_2\mathrm{d}t+\sigma \mathrm{d}w_{2,t},  \quad    x_2(0)=x_{2,0}, 
\end{cases}
\end{equation}
where $w_{i,t}$, $i=1,2$ are independent Brownian motions.  We assume that $\textbf{v}=(v_1,v_2)^{T}$ is divergence-free and mean-zero at any time $t$, i.e.,
\begin{equation}
\label{eqn:divfree}
\nabla\cdot \textbf{v}:=\partial_{x_1}v_1+\partial_{x_2}v_2=0\quad \forall t,
\end{equation}
and
\begin{equation}
\label{eqn:meanzero}
\begin{cases}
\int_{\mathbb{T}}v_1(t,x_1,x_2)\mathrm{d}x_2=0\quad \forall x_1,~t,\\
\int_{\mathbb{T}}v_2(t,x_1,x_2)\mathrm{d}x_1=0\quad \forall x_2,~t,
\end{cases} 
\end{equation}
where $\mathbb{T}=[0,1]$. We also assume that $\textbf{v}$ is smooth and its first-order
derivatives $v_i(t,x_1,x_2)$, $i=1,2$ are bounded. These conditions are necessary to guarantee the
existence and uniqueness of solutions of the SDE \eqref{eqn:particleSDE}; see \cite{Oksendal:13}. Moreover, we assume that the diagonal of the Jacobian of the velocity field $\textbf{v}=(v_1,v_2)^{T}$ are all zeros. A typical example is a Hamiltonian system with a separable Hamiltonian, i.e., there exists $H(t,x_1,x_2)=H_1(t,x_1)+H_2(t,x_2)$ such that, 
\begin{equation}
\label{eqn:Hamiltonian_system}
v_1=-\partial_{x_2}H,\quad v_2=\partial_{x_1}H.
\end{equation}  

\textcolor{black}{
In this paper, we denote with slightly abuse of notation that $v_1(t,x_2)=v_1(t,x_1,x_2)$ and $v_2(t,x_1)=v_2(t,x_1,x_2)$. These notations simplify our derivation. Whenever a statement corresponds to $v_1(t,x_2)$ (or $v_2(t,x_1)$) is made, it is equivalent to that for $v_1(t,x_1,x_2)$ (or $v_2(t,x_1,x_2)$)}. 

In \cite{WangXinZhang:18}, we proposed a stochastic structure-preserving scheme based on a Lie-Trotter splitting scheme to solve the SDE \eqref{eqn:particleSDE}. Specifically, we split the problem \eqref{eqn:particleSDE} into a deterministic subproblem, 
\begin{equation}\label{eqn:particleSDE_1}
\begin{cases}
\mathrm{d}x_1=v_1(t,x_2)\mathrm{d}t, \\
\mathrm{d}x_2=v_2(t,x_1)\mathrm{d}t, 
\end{cases}
\end{equation} 
which is solved by using a symplectic-preserving scheme (e.g., the symplectic Euler scheme for deterministic equations with frozen time), and a stochastic subproblem,  
\begin{equation}\label{eqn:particleSDE_2}
\begin{cases}
\mathrm{d}x_1=\sigma \mathrm{d}w_{1,t},  \\
\mathrm{d}x_2=\sigma \mathrm{d}w_{2,t}, 
\end{cases}
\end{equation} 
which is solved by using the Euler-Maruyama scheme \cite{Oksendal:13}. Notice that when $\sigma$ is
a constant in \eqref{eqn:particleSDE_2}, {\color{black}the Euler-Maruyama scheme exactly solves Eq.\eqref{eqn:particleSDE_2}}

Now we discuss how to discretize Eq.\eqref{eqn:particleSDE}. From time $t=t_n$ to time $t=t_{n+1}$, where $t_{n+1}=t_{n}+\Delta t$, $t_0=0$, and $\Delta t$ is the time step, we assume the numerical solution $\textbf{X}^n=(x_1^n,x_2^n)^{T}$ is given, which approximates the exact solution $\textbf{X}(t_n)$ to the SDE \eqref{eqn:particleSDE} at time $t_n=n\Delta t$. Then, we apply the Lie-Trotter splitting method to solve the SDE \eqref{eqn:particleSDE} and obtain,
\begin{equation}\label{scheme}
\begin{cases}
x_1^{n+1}=x_1^{n}+v_1(t_{n+\frac{1}{2}},x_2^{n})\Delta t+\sigma N^{n}_1,\\
x_2^{n+1}=x_2^{n}+v_2\big(t_{n+\frac{1}{2}},x_1^{n}+v_1(t_{n+\frac{1}{2}},x_2^{n})\Delta t\big)\Delta t+\sigma N^{n}_2,
\end{cases}
\end{equation}
where $t_{n+\frac{1}{2}}=t_n+\frac{\Delta t}{2}$,  $N^{n}_1=\sqrt{\Delta t}\xi_1$, $N^{n}_2=\sqrt{\Delta t}\xi_2$,  and $\xi_1$, $\xi_2\sim\mathcal{N}(0,1)$ are i.i.d. normal random variables. In this paper, we view the solution sequence $\textbf{X}^n=(x_1^n,x_2^n)^{T}$, $n=1,2,3,...$, generated by the scheme \eqref{scheme} as a discrete Markov stochastic process, which enables us to use techniques from stochastic process to obtain a sharp convergence analysis for the numerical solutions; see Section \ref{sec:ConvergenceAnalysis}. 

In a 2D Hamiltonian system, when the system contains an additive temporal noise, the noise itself is considered to be symplectic pathwise \cite{Milstein:02}. Therefore, we state that the scheme \eqref{scheme} is stochastic symplectic-preserving since it preserves symplecticity. Specifically, the scheme \eqref{scheme} can be viewed as a composition of two symplectic transforms. In addition, we know that the numerical solution converges to the exact one as the time step $\Delta t$ approaches zero. In high-dimensional systems, a structure-preserving scheme refers to a volume-preserving scheme; see Section \ref{sec:HighDimensionalCases}.

\subsection{The backward Kolmogorov equation and related results}\label{sec:KnownFacts}
\noindent
We first define the backward Kolmogorov equation associated with the Eq.\eqref{eqn:particleSDE} as 
\begin{align}\label{BackwardKolmolgorovEquation}
u_{t}=\mathcal{L}u, \quad u(0,\tau,\textbf{x})=\phi(\tau,\textbf{x}),
\end{align}
where the generator $\mathcal{L}$ associated with the Markov process in Eq.\eqref{eqn:particleSDE} 
is given by
\begin{align}
\mathcal{L}=\partial_\tau+v_1(\tau,x_2)\partial_{x_1}+v_2(\tau,x_1)\partial_{x_2}+
\frac{\sigma^2}{2}(\partial_{x_1x_1}+\partial_{x_2x_2}).
\label{HamiltonianFlowOperator}
\end{align}
Recall that the solution $u(t,\tau,\textbf{x})$ to the Eq.\eqref{BackwardKolmolgorovEquation} satisfies,
$u(t,\tau,\textbf{x})=\mathbb{E}\big[\phi(t+\tau,\textbf{X}_{t+\tau})|\textbf{X}_\tau=\textbf{x}\big]$ where $\textbf{X}_t$ is the solution to Eq.\eqref{eqn:particleSDE} and $\phi$ is a smooth function in $\mathbb{R}^{1}\times \mathbb{R}^{2}$. 
In other words, $u(t,\tau,\textbf{x})$ is the flow generated by the original SDE \eqref{eqn:particleSDE}.

Similarly, 
 we can study the flow generated by the stochastic structure-preserving scheme \eqref{scheme}.
According to the splitting method used in the derivation of the scheme in Section \ref{sec:DerivationSchemes}, we respectively define $\mathcal{L}_1=\partial_\tau$, $\mathcal{L}_2=v_1\partial_{x_1}$, $\mathcal{L}_3=v_2\partial_{x_2}$, and  $\mathcal{L}_4=\frac{\sigma^2}{2}(\partial_{x_1x_1}+\partial_{x_2x_2})$. Starting from $u(0,\cdot,\cdot)$, during one time step $\Delta t$, we compute 
\begin{equation}\label{eqn:flowpde}
\begin{cases}
\partial_tu^1&=\mathcal{L}_1 u^1,\quad u^1(0,\cdot,\cdot)=u(0,\cdot,\cdot),\\
\partial_t u^2&=\mathcal{L}_2 u^2,\quad u^2(0,\cdot,\cdot)=u^1(\frac{\Delta t}{2},\cdot,\cdot), \\
\partial_t u^3&=\mathcal{L}_3u^3,\quad u^3(0,\cdot,\cdot)=u^2(\Delta t,\cdot,\cdot),\\
\partial_t u^4&=\mathcal{L}_1u^4,\quad u^4(0,\cdot,\cdot)=u^3(\Delta t,\cdot,\cdot),\\
\partial_t u^5&=\mathcal{L}_4u^5,\quad u^5(0,\cdot,\cdot)=u^4(\frac{\Delta t}{2},\cdot,\cdot).\\
\end{cases}		
\end{equation}
Then, $u^5(\Delta t,\cdot,\cdot)$ will be the flow at time $t=\Delta t$ generated by our stochastic structure-preserving scheme \eqref{scheme} and it approximates the solution $u(\Delta t,\cdot,\cdot)$ to the Eq.\eqref{BackwardKolmolgorovEquation} well when $\Delta t$ is small. It is also worth mentioning that, $u^3(\Delta t,\cdot,\cdot)$ is the exact flow generated by the deterministic symplectic Euler scheme in solving Eq.\eqref{eqn:particleSDE_1}. 
We repeat this process to compute the flow equations of our scheme at other time steps, which approximate the solution $u(n\Delta t,\cdot,\cdot), n=2,3,...$ to the Eq.\eqref{BackwardKolmolgorovEquation} at different time steps. 
\begin{remark} 
Given the operators $\mathcal{L}_i$, $i=1,2,3,4$, there are many possible choices in setting the coefficients for each operator $\mathcal{L}_i$ and designing the splitting method; see Section 2.5 of \cite{ErnstLubich:06}. Eq.\eqref{eqn:flowpde} is a simple choice that was used in this paper.	
\end{remark}

To analyze the error between the flow operator in Eq.\eqref{BackwardKolmolgorovEquation} and the composition of  operators in Eq.\eqref{eqn:flowpde}, we shall resort to the Baker-Campbell-Hausdorff (BCH) formula, which is widely used in {\color{black}non-commutative} algebra \cite{BCHformula1974baker}. For example, in the matrix theory,
\begin{equation}\label{eqn:BCHformula}
\exp(tA)\exp(tB)=\exp\bigg(t(A+B)+t^2\frac{[A,B]}{2}+\frac{t^3}{12}\Big(\big[A,[A,B]\big]+\big[B,[B,A]\big]\Big)+\cdots\bigg),
\end{equation}
where $t$ is a scalar, $A$ and $B$ are two square matrices of the same size, $[,]$ is the Lie-Bracket,
and the remaining terms on the right hand side are all nested Lie-brackets. 

In our analysis, we replace the matrices in Eq.\eqref{eqn:BCHformula} by differential operators and the BCH formula yields critical insights into the particular structure of the splitting error. Let $I_{\Delta t}$ denote the composite flow operator associated with Eq.\eqref{eqn:flowpde}, i.e., 
\begin{equation}\label{eqn:operatorflow}
I_{\Delta t} u(0,\cdot,\cdot):=\exp(\Delta t\mathcal{L}_4)\exp(\frac{\Delta t}{2}\mathcal{L}_1)\exp(\Delta t\mathcal{L}_3)\exp(\Delta t\mathcal{L}_2)\exp(\frac{\Delta t}{2}\mathcal{L}_1)u(0 ,\cdot,\cdot).
\end{equation}
Notice that after propagating time $t=\Delta t$, the exact solution to the Eq.\eqref{BackwardKolmolgorovEquation} started at any $\tau$ can be represented as 
\begin{equation}\label{eqn:TrueFlowOperator}
u(\Delta t,\cdot,\cdot)=\exp(\Delta t\mathcal{L})u(0,\cdot,\cdot)=
\exp\big(\Delta t(\mathcal{L}_1+\mathcal{L}_2+\mathcal{L}_3+\mathcal{L}_4)\big)u(0,\cdot,\cdot).
\end{equation}
Therefore, we can apply the BCH formula to analyze the error between the original flow and 
the approximated flow. Moreover, we find that computing the $k$-th order modified equation associated with Eq.\eqref{eqn:particleSDE} in the backward error analysis (BEA) \cite{Reich:99,debussche2012weak} is equivalent to computing the terms of BCH formula up to order $(\Delta t)^k$ in the Eq.\eqref{eqn:operatorflow}. 
To show that the solution generated by Eq.\eqref{scheme} follows a perturbed Hamiltonian system (with divergence-free velocity and additive noise) at any order $k$, we only need to consider the $(k+1)$-nested Lie bracket consisting of $\big\{\partial_\tau,v_1\partial_{x_1},\ v_2\partial_{x_2},\ \frac{\sigma^2}{2}(\partial_{x_1x_1}+\partial_{x_2x_2})\big\}$ and we can easily see that they generate divergence-free fields.

We remark that given any explicit splitting scheme for deterministic systems, by adding additive noise we shall obtain a similar form of flow propagation. And we shall see in later proof that, such operator formulation is very effective in analyzing the order of convergence and volume-preserving property.


\section{Convergence analysis}\label{sec:ConvergenceAnalysis}
\noindent 
In this section, we prove the convergence rate of our stochastic structure-preserving schemes in computing effective diffusivity based on a probabilistic approach, which allows us to get rid of the exponential growth factor in the error estimate. {\color{black}We first limit our analysis to 2D separable Hamiltonian velocity fields. Then, in Section \ref{sec:HighDimensionalCases} we will show that all the derivations can be generalized to high-dimensional cases.}
\subsection{Convergence to an invariant measure}\label{sec:InvariantMeasure}
\noindent 
To compute the effective diffusivity of a passive tracer model using a Lagrangian numerical scheme is closely related to study the limit of a solution sequence (a stochastic process) generated by the numerical scheme. Therefore, we can apply the results from ergodic theory to study the convergence behaviors of the solution. 

Let $(S,\Sigma)$ be a probability space, on which a family $P(\textbf{x},E)$, $\textbf{x}\in S$, $E\in \Sigma$, of probability measure is defined. We assume $\textbf{x}\to P(\textbf{x},E)$ is measurable, $\forall E\in \Sigma$. This corresponds to a linear bounded operator on $\mathcal{B}(S)$, \textcolor{black}{which is  the space of bounded measurable functions on $S$}. This operator, denoted by $P$, is defined by,
\begin{equation}
P\phi(\textbf{x})=\int_{S}P(\textbf{x},\mathrm{d}\textbf{z})\phi(\textbf{z}), \quad \forall \phi\in \mathcal{B}(S).   
\end{equation}
Clearly $||P||\leq 1$. One of the main objectives of ergodic theory is to study the limit of the operator sequence $P^n$ as $n\rightarrow +\infty$. The result can be summarized into the following proposition, 
which plays a fundamental role in our convergence analysis. 
\begin{proposition}[Theorem 3.3.1 of \cite{BensoussanLionsPapa:2011}]\label{pro:doobs}
We assume that,
	\begin{enumerate}
		\item $S$ is a compact metric space and $\Sigma$ is the Borel $\sigma$-algebra;
		\item there exists a probability measure $\mu$ on $(S,\Sigma)$ such that $P(\textbf{x},E)=\int_E p(\textbf{x},\textbf{y})\mu(\mathrm{d}\textbf{y})$;
		\item $p(\textbf{x},\textbf{y}):S\times S\to \mathbb{R}^+$ is continuous;
		\item there exists a ball $U_0$ such that $\mu(U_0)>0$ and a positive number $\delta>0$ (depending on $U_0$) such that $p(\textbf{x},\textbf{y})\geq \delta$, $\textbf{x}\in S$, $\forall \textbf{y}\in U_0$.
	\end{enumerate}
Then, there exists one and only one invariant probability measure $\pi$ on $(S,\Sigma)$ and one has,
\begin{equation}
\sup_{\textbf{x}\in S}\Big|P^n\phi(\textbf{x})-\int\phi\pi(\mathrm{d}\textbf{x})\Big|\leq C||\phi||e^{-\rho n},\ \forall \phi\in \mathcal{B}(S),
\end{equation}
where $\rho=\log\frac{1}{1-\delta \mu(U_0)}>0$ and $C=\frac{2}{1-\delta \mu(U_0)}>0$ are independent of $\phi$.
\end{proposition}
Now we study the convergence behaviors of the solution generated by our stochastic structure-preserving scheme \eqref{scheme}. We first prove a lemma as follows. 
\begin{lemma}\label{lem:kernel}
	Let $\tilde{Y}=\mathbb{R}^2/\mathbb{Z}^2$ denote the physical torus space and $\mathbb{T}$ be the time periodic space. Let $I_{\tau,1+\tau}^*$ denote the  transform of the density  on $\tilde{Y}$ during $[\tau,1+\tau]$ (time period is $1$) using the numerical scheme \eqref{scheme}.	In addition, let $I_{\tau,1+\tau}$ denote the adjoint operator (i.e., the flow operator) of $I_{\tau,1+\tau}^*$ in the space of $\mathcal{B}(\tilde{Y})$, which is the set of bounded measurable functions on $\tilde{Y}$.  Then, there exists one and only one invariant probability measure on $(\tilde{Y},\Sigma)$, denoted by $\pi_\tau$, satisfying,
	\begin{equation}
	\sup_{\textbf{x}\in\tilde{Y}}\Big|\big((I_{\tau,1+\tau})^n\phi\big) (\textbf{x})-\int \phi(\textbf{x}')\pi_\tau(\mathrm{d}\textbf{x}')\Big|\leq C ||\phi||_{L_\infty}e^{-\rho n},\quad \forall \phi\in \mathcal{B}(\tilde{Y}),
	\label{kernel-estimate1}
	\end{equation}
	where $\rho>0$, $C>0$ are independent of $\phi(\cdot)$. Moreover, the kernel space of $(I_d-I_{\tau,1+\tau})$ is the constant functions in $\tilde{Y}$, where $I_d$ is the identity operator.

	\begin{proof}
	
	We shall verify that the transition kernel associated with the numerical scheme \eqref{scheme} 
	satisfies the assumptions required by Prop. \ref{pro:doobs}.
	First notice that in the space $\mathbb{R}^2$, the integration process associated with the numerical scheme \eqref{scheme} 	can be expressed as a Markov process with the transition kernel,
	\begin{align}\label{eqn:rnkernel}
	&K_{t}\big(\textbf{X}^n,\textbf{X}^{n+1}\big)=\frac{1}{2\pi \sigma^2\Delta t}\cdot\nonumber\\&\resizebox{.9\hsize}{!}{$\exp\Bigg(
		-\frac{\Big(x_1^{n+1}-x_1^{n}-v_1(t+\frac{\Delta t}{2},x_2^{n})\Delta t\Big)^2+\Big(x_2^{n+1}-x_2^{n}-v_2\big(t+\frac{\Delta t}{2},x^{n+1}_1-x^{n}_1-v_1(t+\frac{\Delta t}{2},x^{n}_2)\Delta t\big)\Delta t\Big)^2}{2\sigma^2\Delta t}\Bigg)$},
	\end{align}	
	where $\textbf{X}^n=(x^{n}_1,x^{n}_2)^{T}$ and $\textbf{X}^{n+1}=(x^{n+1}_1,x^{n+1}_2)^{T}$ are the numerical solutions at time $t=t_n$ and $t=t_{n+1}$, respectively. 
	
	Then, using the periodicity of $\textbf{v}$, we directly extend Eq.\eqref{eqn:rnkernel} to the torus space $\tilde{Y}$ as 
	\begin{align}\label{eqn:tnkernel}
	&\tilde{K}_{\tau}\big(\textbf{X}^n,\textbf{X}^{n+1}\big)=\sum_{i,j\in\mathbb{Z}}\frac{1}{2\pi \sigma^2\Delta t}\cdot\nonumber\\&\resizebox{.9\hsize}{!}{$\exp\Bigg(
		-\frac{\Big(x^{n+1}_1+i-x^{n}_1-v_1(\tau+\frac{\Delta t}{2},x^{n}_2)\Delta t\Big)^2+\Big(x^{n+1}_2+j-x^{n}_2-v_2\big(\tau+\frac{\Delta t}{2},x^{n+1}_1-x^{n}_1-v_1(\tau+\frac{\Delta t}{2},x^{n}_2)\Delta t\big)\Delta t\Big)^2}{2\sigma^2\Delta t}\Bigg)$}.
	\end{align}
	Let $\tilde{\textbf{K}}_{\tau,\tau+k\Delta t}$ denote the kernel from $\tau$ to $\tau+k\Delta t$, 
	which is the density of the transition kernel associated with applying our scheme starting from time $\tau$ for $k$ steps. Then, we have
	\begin{equation}
	\tilde{\textbf{K}}_{\tau,\tau+k\Delta t}(\textbf{X}^0,\textbf{X}^k)=\int_{(\tilde{Y})^{k-1}}\prod_{m=0}^{k-1}\tilde{K}_{\tau+m\Delta t}(\textbf{X}^m,\textbf{X}^{m+1})\mathrm{d}\textbf{X}^1\mathrm{d}\textbf{X}^2\cdots \mathrm{d}\textbf{X}^{k-1}.
	\label{eqn:defineKtauktau}
	\end{equation} 
    We choose $k=\frac{1}{\Delta t}$ and obtain $\tilde{\textbf{K}}_{\tau,\tau+1}$. 
	One can see that the kernel $\tilde{\textbf{K}}_{\tau,\tau+1}$ is essentially bounded above zero since $\tilde{K}_{\tau+m\Delta t}$ in \eqref{eqn:defineKtauktau} are all positive.  Moreover, if $0 < \Delta t\ll1$, $\tilde{\textbf{K}}_{\tau,\tau+1}$ is a continuous function on the domain $\tilde{Y}\times\tilde{Y}$.  Then by noticing that the domain $\tilde{Y}\times\tilde{Y}$ is compact, the kernel $\tilde{\textbf{K}}_{\tau,\tau+1}$ is strictly positive.  Namely, there exists $\delta_{\tau}>0$ such that $\tilde{\textbf{K}}_{\tau,\tau+1}(\textbf{X}^0,\textbf{X}^k)>\delta_{\tau},\ \forall (\textbf{X}^0,\textbf{X}^{k})\in \tilde{Y}\times\tilde{Y}$.  If we apply Prop.\ref{pro:doobs} to $I_{\tau,1+\tau}$ (whose kernel is $\tilde{\textbf{K}}_{\tau,\tau+1}$), we prove the  statement in \eqref{kernel-estimate1}. 
	
	Finally, we know that the operator $I_{\tau,1+\tau}$ is compact since it is an integral operator with a continuous kernel. By using the Fredholm alternative, we know that   $\dim\ker(I_d-I_{\tau,1+\tau})=\dim\ker(I_d-I_{\tau,1+\tau}^*)=1$. Therefore, it is easy to verify that the constant functions are in the kernel of $(I_d-I_{\tau,1+\tau})$. 
\end{proof}
\end{lemma}
Equipped with the Lemma \ref{lem:kernel}, we study the convergence rate  of the space-time transition kernel associated with our numerical scheme \eqref{scheme}. 
\begin{theorem}	\label{thm:main-per}
	Let $\Delta t=\frac{1}{N}$, $N$ is a positive integer.
	We have the following properties hold:
	\begin{itemize}
		\item[(a)] Given $\Delta t$, there exists $C>0$ and $\rho>0$, such that,
		\begin{equation}
		\sup_{\tau,\textbf{x}}\Big|\big(I_{\Delta t}^N\big)^n\phi (\tau,\textbf{x})-\int \phi(\tau,\textbf{x}')\pi_\tau(\mathrm{d}\textbf{x}')\Big|\leq C ||\phi||_{L_\infty}e^{-\rho n},\quad \forall \phi\in \mathcal{B}(\mathbb{T}\times\tilde{Y}),
		\end{equation}
		where $C$ and $\rho$ do not depend on $\phi $ and $\tau$. 
		\item[(b)] If $\int_{\tilde{Y}} \phi\pi_\tau=0$, then we get
		\begin{equation}
		\lim_{n\to\infty}\sum_{i=1}^n\mathbb{E}\phi(\tau,\textbf{X}^{N\tau+i})<\infty, \quad \forall \tau\in\mathbb{T}.
		\end{equation}
		\item[(c)] The kernel space of $(I_d-I_{\Delta t}^N)$ is $\big\{c(\tau)~|~c(\tau)\text{ 
			is a periodic function in}~\mathbb{T}~\text{with period 1}\big \}$.
	\end{itemize}
 	
\begin{proof}
	By definition of $I_{\Delta t}$ and $I_{\tau,1+\tau}$ in Eq.\eqref{eqn:operatorflow} and Lemma \ref{lem:kernel}, we have $(I_{\Delta t})^N\phi(\tau,\cdot)\equiv I_{\tau,1+\tau}\phi(\tau,\cdot)$. 
	To prove the property (a), we need to show that the lower bound of the kernel $\tilde{\textbf{K}}_{\tau,\tau+1}$, which is defined 
	in the proof of Lemma \ref{lem:kernel}, does not depend on $\tau$.  For all $\tau\in\mathbb{T}$,
	$\textbf{X}^n=(x^{n}_1,x^{n}_2)^{T}\in\mathbb{T}^2$ and $\textbf{X}^{n+1}=(x^{n+1}_1,x^{n+1}_2)^{T}\in\mathbb{T}^2$, we pick $i_0=\lfloor -x^{n+1}_1+x^n_1+v_1(\tau+\frac{\Delta t}{2},x^n_2)\Delta t\rfloor$ and $j_0=\lfloor -x^{n+1}_2+x^n_2+v_2\big(\tau+\frac{\Delta t}{2},x^{n+1}_1-x^{n}_1-v_1(\tau+\frac{\Delta t}{2},x^n_2)\Delta t\big)\Delta t\rfloor$, where $\lfloor a\rfloor$ denotes the largest integer not greater than $a$. Applying to Eq.\eqref{eqn:tnkernel}, we can see that 
	\begin{align}
	&\tilde{K}_{\tau}\big(\textbf{X}^n,\textbf{X}^{n+1}\big)\geq \frac{1}{2\pi \sigma^2\Delta t}\cdot\nonumber\\&\resizebox{.9\hsize}{!}{$\exp\Bigg(
		-\frac{\Big(x^{n+1}_1+i_0-x^n_1-v_1(\tau+\frac{\Delta t}{2},x^n_2)\Delta t\Big)^2+\Big(x^{n+1}_2+j_0-x^n_2-v_2\big(\tau+\frac{\Delta t}{2},x^{n+1}_1-x^n_1-v_1(\tau+\frac{\Delta t}{2},x^n_2)\Delta t\big)\Delta t\Big)^2}{2\sigma^2\Delta t}\Bigg)$}\nonumber \\
	&\geq \frac{1}{2\pi \sigma^2\Delta t}\exp\big(-\frac{1}{\sigma^2\Delta t}\big)>0.
	\end{align}  
	According to the definition of the kernel $\tilde{\textbf{K}}_{\tau,\tau+1}$; see Eq.\eqref{eqn:defineKtauktau}, we know the minimal value of $\tilde{\textbf{K}}_{\tau,\tau+1}$ is above zero and is independent of $\tau$. Now, we apply this observation to Lemma \ref{lem:kernel} and conclude the proof of the property (a). The property (b) is a simple conclusion of the exponential decay property 
	proved in (a). For the property (c), we consider the equation $ I_{\Delta t}^N w=w$. Then, for a given time $\tau$, we have $I_{\tau,1+\tau}w(\tau,\cdot)=w(\tau,\cdot)$. Notice the fact that in Lemma \ref{lem:kernel} the invariant space of $I_{\tau,1+\tau}$ is constant in the spacial variable. Thus, we obtain $w=w(\tau)$.
\end{proof}
\end{theorem}

Before we close this subsection, we provide a convergence result for the inverse of operator sequences, 
which will be useful in our convergence analysis.
\begin{proposition}\label{prop:inverseoperator}
Let $X,Y$ denote two Banach spaces. Assume $T_{n }$, $T$ are bounded linear operators from  $X$ to $Y$, satisfying $\lim_{n \to \infty}||T_{n }-T||_{\mathcal{B}(X,Y)}=0$, and $T^{-1}\in\mathcal{B}(Y,X)$. Given $f\in Y$, if $T_{n}^{-1}f$, $n=1,2,...$ uniquely exist, then we have a convergence estimate as follows,
	\begin{align}\label{eqn:est_inverseoperator}
	\lim_{n \to \infty}\big|\big|(T_{n}^{-1}-T^{-1})f\big|\big|=0.
	\end{align}
\end{proposition}
The proof is quite standard. It can also be viewed as a modification of Theorem 1.16 in Section IV of \cite{kato2013perturbation}. 
 

\subsection{A discrete-type cell problem}\label{sec:ProbabilisticProof}
\noindent 
In the Eulerian framework, the periodic solution of the cell problem \eqref{CellProblem_EffectiveDiffusivity} and the corresponding formula for the effective diffusivity \eqref{Def_EffectiveDiffusivity_Euler} play a key role in studying the behaviors of chaotic and stochastic flows. In the Lagrangian framework,
we shall define a discrete analogue of  cell problem that enables us to compute the effective diffusivity. 
Let $\textbf{X}^0=(x_{1}^{0},x_{2}^{0})^T$ be the initial data and $\textbf{X}^n=(x_{1}^{n},x_{2}^{n})^T$ denote the numerical solution generated by the scheme \eqref{scheme} at $t_n=n\Delta t$, i.e. 
\begin{equation}\label{scheme_analysis}
\begin{cases}
x_1^{n}=x_1^{n-1}+v_1(t_{n-\frac{1}{2}},x_2^{n-1})\Delta t+\sigma N^{n-1}_1,\\
x_2^{n}=x_2^{n-1}+v_2\big(t_{n-\frac{1}{2}},x_1^{n-1}+v_1(t_{n-\frac{1}{2}},x_2^{n-1})\Delta t\big)\Delta t+\sigma N^{n-1}_2,
\end{cases}
\end{equation}
where $N^{n-1}_1=\sqrt{\Delta t}\xi_1$, $N^{n-1}_2=\sqrt{\Delta t}\xi_2$, and $\xi_1$, $\xi_2\sim\mathcal{N}(0,1)$ are i.i.d. normal random variables. For convenience we have replaced $n+1$ by $n$. 

First of all, we show that the solutions $x_1^n$ and $x^n_2$ obtained by the scheme \eqref{scheme_analysis} have bounded expectations if the initial values are bounded. Taking expectation of the first equation of Eq.\eqref{scheme_analysis} on both sides, we obtain 
\begin{align}
\mathbb{E}x_1^n=\mathbb{E}x_1^{n-1}+\Delta t \mathbb{E}v_1(t_{n-\frac{1}{2}},x_2^{n-1})=\mathbb{E}x_1^0+\Delta t\sum_{k=0}^{n-1}\mathbb{E}v_1(t_{k+\frac{1}{2}},x_2^{k}).
\label{bounded_Ep}
\end{align}
As a symplectic scheme in 2D, the numerical scheme \eqref{scheme_analysis} admits the uniform measure as its invariant measure.  Applying the results (a) and (b) of Theorem \ref{thm:main-per} and using the fact that $\textbf{v}$ is a periodic function with zero mean, we know that,
\begin{equation}\label{eqn:decay_expactation}
\sup_{\textbf{X}^0\in\tilde{Y}}|\mathbb{E}v_1(t_{k+\frac{1}{2}},\textbf{X}^{k})|\leq e^{-\rho k}C_N\sup_{m=1,2,...,N,\ \textbf{x}\in\mathbb{T}^2}||v_1(t_{m+\frac{1}{2}},\textbf{x})||_\infty.
\end{equation} 
Notice that $v_1(t_{k+\frac{1}{2}},\textbf{X}^{k})$ is equivalent to $v_1(t_{k+\frac{1}{2}},x_2^{k})$, since $v_1$ is indpentdent of $x_1^{k}$. By applying triangle inequalities in Eq.\eqref{bounded_Ep} and using the result in  Eq.\eqref{eqn:decay_expactation}, we arrive at,
\begin{equation}
   |\mathbb{E}x_1^n|\leq|\mathbb{E}x_1^0|+C_1||v_1||_\infty, \label{bounded_Ep2}
\end{equation}
where $C_1$ does not depend on $n$. Using the same approach, we know that expectation of the second component $\mathbb{E}x^n_2$ is also bounded.
	
Now, we are in the position to define the discrete-type cell problem. Starting at time $\tau$ with time step $\Delta t=\frac{1}{N}$, we denote the starting time index to be $N\tau$. Then, we define
\begin{equation}
\hat{v}_{1,N}(\tau,\textbf{x})=\Delta t\sum_{i=0}^{\infty}\mathbb{E}\big[v_1(t_{i+\frac{1}{2}}+\tau,\textbf{X}^{N\tau+i})|\textbf{X}^{N\tau}=\textbf{x}\big],
\label{discrete--cell-problem}
\end{equation}
where the summation is well defined due to the fact stated in Eq.\eqref{eqn:decay_expactation}. We will show that $\hat{v}_{1,N}(\tau,\textbf{x})$ satisfies the following properties. Namely, $\hat{v}_{1,N}(\tau,\textbf{x})$ is the solution of the discrete-type cell problem defined in Eq.\eqref{eqn:DiscreteCellProb}.
\begin{lemma}\label{lem:existenceofcellpro}
According to our assumption on $\textbf{v}$, we know that $v_1$ is a periodic function with zero mean on $\tilde{Y}$, $\forall \tau$, i.e., $\int_{\tilde{Y}}v_1=0$. Therefore, $\hat{v}_{1,N}(\tau,\textbf{x})$ is the unique solution in $\mathcal{B}_0(\mathbb{T}\times\tilde{Y})$ such that 
\begin{equation}\label{eqn:DiscreteCellProb}
\hat{v}_{1,N}(\tau,\textbf{x})=(I_{\Delta t}\hat{v}_{1,N})(\tau,\textbf{x})+\Delta tv_1(\tau+\frac{\Delta t}{2},\textbf{x}),
\end{equation}
where $\Delta t=\frac{1}{N}$ and the operator $I_{\Delta t}$ is defined in \eqref{eqn:operatorflow}. Moreover, $\hat{v}_{1,N}(\tau,\textbf{x})$ is smooth.
\begin{proof}
Throughout the proof, 
we shall use the fact that if $X$, $Y$ are random processes and $Y$ is measurable under a filtration $\mathcal{F}$, then with appropriate integrability assumption, we have 
		\begin{equation}
		\mathbb{E}[XY]=\mathbb{E}\Big[\mathbb{E}[XY|\mathcal{F}]\Big]=\mathbb{E}\Big[\mathbb{E}[X|\mathcal{F}]Y\Big].
		\end{equation}
		Some simple calculations will give that 
		\begin{align}
		\hat{v}_{1,N}(\tau,\textbf{x})-\Delta tv_1(\tau+\frac{\Delta t}{2},\textbf{x})=&\Delta t\sum_{i=1}^{\infty}\mathbb{E}\big[v_1(t_{i+\frac{1}{2}}+\tau,\textbf{X}^{N\tau+i})|\textbf{X}^{N\tau}=\textbf{x}\big]\nonumber\\
		=&\mathbb{E}\Big[\Delta t\sum_{i=1}^{\infty}\mathbb{E}\big[v_1(t_{i+\frac{1}{2}}+\tau,\textbf{X}^{N\tau+i})|\textbf{X}^{N\tau+1}\big]|\textbf{X}^{N\tau}=\textbf{x}\Big] \nonumber\\=& \mathbb{E}\big[\hat{v}_{1,N}(\tau+\Delta t,\textbf{X}^{N\tau+1})|\textbf{X}^{N\tau}=\textbf{x}\big].
		\label{DiscreteCellProblem1}
		\end{align}	
		Recall the definition of the operator $I_{\Delta t}$ in \eqref{eqn:operatorflow}, Eq.\eqref{DiscreteCellProblem1}
		implies that  
		\begin{equation}\label{cell_eqn}
		\hat{v}_{1,N}(\tau,\textbf{x})-\Delta tv_{1}(\tau+\frac{\Delta t}{2},\textbf{x})=(I_{\Delta t}\hat{v}_{1,N})(\tau,\textbf{x}).
		\end{equation}
		
Suppose we have that $I_{\Delta t}w=w$. Then, we get $(I_{\Delta t})^Nw=w$. According to Theorem \ref{thm:main-per}, we know that $w=0$ if $\int_{\tilde{Y}}w\mathrm{d}\textbf{x}=0$, $\forall t$. So $\ker(I_{\Delta t}-I_d)=\{0\}$ and $\hat{v}_{1,N}$ is unique.
Finally, by the definition of $\hat{v}_{1,N}$, we obtain that 
\begin{align}
\hat{v}_{1,N}(\tau,\textbf{x})=&\Delta t\sum_{i=0}^{\infty}\mathbb{E}\big[v_1(t_{i+\frac{1}{2}}+\tau,\textbf{X}^{N\tau+i})|\textbf{X}^{N\tau}=\textbf{x}\big]\nonumber\\
=&\Delta t\sum_{i=0}^{\infty} \int_{\tilde{Y}}v_1(t_{i+\frac{1}{2}}+\tau,\textbf{y})\tilde{K}_{\tau,\tau+i\Delta t} (\textbf{x},\textbf{y})\mathrm{d}\textbf{y},
\end{align}
which indicates that $\hat{v}_{1,N}$ has the same regularity as $v_1$ does. Notice  
the kernel $\tilde{K}_{\tau,\tau+i\Delta t} (\textbf{x},\textbf{y})$ has a fast decay property, which 
guarantees the order of the differentiation and  summation is interchangeable.  
	\end{proof}
\end{lemma}
\begin{remark}
Notice that $v_1$ and $\hat{v}_{1,N}$ only depend on the second component of the numerical solution $\textbf{X}^n=(x_{1}^{n},x_{2}^{n})^T$. However, we will write $v_1$ and $\hat{v}_{1,N}$ as functions of $\textbf{X}^n$ when we view $\textbf{X}^n$ as a Markov process in the convergence analysis.  
\end{remark}	
\begin{remark}
When the flow is time-independent, we obtain 
\begin{equation}
\mathbb{E}[\hat{v}_{1,N}(\textbf{X}^{n+1})|\textbf{X}^n]-\hat{v}_{1,N}(\textbf{X}^n)=-\Delta tv_1(\textbf{X}^n),\quad a.s.\quad\forall n\in \mathbb{N}.
\end{equation}
Therefore, the discrete-type cell problem defined in \eqref{eqn:DiscreteCellProb} is a generalization of 
the discrete-type cell problem for time-independent flow problems studied in our previous work \cite{Zhongjian2018sharp}, although technically it is more involved.
\end{remark}


In the remaining part of this paper, we only need the result that $\hat{v}_{1,N}(\tau,\textbf{x})$ is unique in an H\"{o}lder space $\mathbb{C}^{p_1,p_2,\alpha}_0(\mathbb{T}\times\tilde{Y})\subsetneq\mathcal{B}(\mathbb{T}\times\tilde{Y})$. To be precise, given a smooth drift function $v_1$, $\hat{v}_{1,N}(\tau,\textbf{x})$ will be in $\mathbb{C}^{p_1,p_2,\alpha}_0(\tilde{Y})$, where $p_1\geq 2, p_2\geq 6, 0<\alpha<1$ and the subscript index $0$ indicates that it is a subspace with zero-mean functions. 


\subsection{Convergence estimate of the discrete-type cell problem}\label{sec:ConvergenceDiscreteProblem} 
\noindent
In this section, we shall prove that the solution $\hat{v}_{1,N}(\tau,\textbf{x})$ of the discrete-type cell problem (i.e., Eq.\eqref{eqn:DiscreteCellProb}) converges to the solution of a continuous cell problem in certain subspace. Here, we choose the space $\mathbb{C}^{2,6,\alpha}_0(\mathbb{T}^1\times\tilde{Y})$ to carry out our analysis. However, there is no requirement that we have to choose this one. In fact, any space that has certain regularity (belongs to the domain of the operator $\mathcal{L}$) will work. 
Notice that the continuous cell problem \eqref{CellProblem_EffectiveDiffusivity} is defined for a vector function, where the first component satisfies   
\begin{equation}\label{continuous_cellproblem}
\mathcal{L}\chi_1=-v_1. 
\end{equation}
For the two-dimensional problem, the operator $\mathcal{L}$ is defined in Eq.\eqref{HamiltonianFlowOperator}. 
Given the fact that $v_1$ is a smooth function defined on $\mathbb{T}^1\times\tilde{Y}$, which satisfies $\int_{\tilde{Y}}v_1(\tau,\textbf{x})\mathrm{d}\textbf{x}=0,\ \forall\tau\in\mathbb{T}^1$. Then, Eq.\eqref{continuous_cellproblem} admits a unique solution $\chi_1$ in $\mathbb{C}_0^{2,6,\alpha}(\mathbb{T}^1\times\tilde{Y})$. This is a standard result of parabolic PDEs in H\"{o}lder space (see, e.g., the Theorem 8.7.3 in \cite{krylov1996lectures}). 
The following theorem states that under certain conditions the solution of the discrete-type cell problem
converges to the solution of the continuous one.  
\begin{theorem}	\label{thm:disc-to-cont}
Assume $v_1$ is a smooth function defined on $\mathbb{T}^1\times\tilde{Y}$, satisfying $\int_{\tilde{Y}}v_1(\tau,\textbf{x})\mathrm{d}\textbf{x}=0,\ \forall\tau\in\mathbb{T}^1$. 
Let $\hat{v}_1$ and $\chi_1$ be the solutions of the discrete-type cell problem \eqref{eqn:DiscreteCellProb} and continuous cell problem \eqref{continuous_cellproblem}, respectively. Then, we have the following convergence estimate holds
\begin{equation}\label{ConvergenceResult_ftochi}
||\chi_1-\hat{v}_1||=\mathcal{O}(\Delta t),
\end{equation}
where $||\cdot||$ is a function norm associated with the space $\mathbb{C}_0^{2,6,\alpha}(\mathbb{T}^1\times\tilde{Y})$.
\end{theorem}
		
\begin{proof}
Using Prop. \ref{prop:inverseoperator}, one can easily verify that $\mathcal{L}$ is a bijection between two Banach spaces $\mathbb{C}_0^{2,6,\alpha}(\mathbb{T}^1\times\tilde{Y})$ and $\mathbb{C}_0^{1,4,\alpha}(\mathbb{T}^1\times\tilde{Y})$ and its inverse is bounded.
Integrating Eq.\eqref{continuous_cellproblem} along time gives, 
\begin{align}\label{eqn:asymcell}
\exp (\Delta t \mathcal{L})\chi_1-\chi_1=-v_1\Delta t +\mathcal{O}\big((\Delta t)^2\big)\equiv-\Delta t\bar{v}_1,
\end{align}
where $\bar{v}_1=v_1+O(\Delta t)$. Combining Eqns.\eqref{cell_eqn} and \eqref{eqn:asymcell}, we obtain
\begin{equation}\label{eqn:derivationform}
\exp (\Delta t \mathcal{L})\chi_1 - I_{\Delta t} \hat{v}_1 
-(\chi_1 - \hat{v}_1  ) = \Delta t(v_1-\bar{v}_1).
\end{equation}
Notice that Eq.\eqref{eqn:derivationform} shows the connection between $\chi_1$ and $\hat{v}_1$. 
After some simple calculations, we get that  
\begin{equation}\label{eqn:balancedform}
\mathcal{L}(\chi_1-\hat{v}_1)=(\mathcal{L}-\tilde{L}_1)(\chi_1-\hat{v}_1)+\tilde{L}_2\hat{v}_1+(v_1-\bar{v}_1),
\end{equation}
where
\begin{equation}
\tilde{L}_1=\frac{\exp (\Delta t \mathcal{L})-I_d}{\Delta t},
\quad\text{and} \quad 
\tilde{L}_2=\frac{I_{\Delta t}-\exp (\Delta t \mathcal{L})}{\Delta t}.
\end{equation}
Moreover, we can verify that in the space of bounded linear operators from $\mathbb{C}_0^{2,6,\alpha}(\tilde{Y})$ to $\mathbb{C}_0^{1,4,\alpha}(\tilde{Y})$, there is a strong convergence in the operator norm $||\cdot||$, 
\begin{equation}\label{conv:operatorstrong}
||\mathcal{L}-\tilde{L}_1||=\mathcal{O}(\Delta t)\quad \text{as } \Delta t\to 0.
\end{equation}
For the operator $\tilde{L}_2$, noticing that $\mathcal{L}=\mathcal{L}_1+\mathcal{L}_2+\mathcal{L}_3+\mathcal{L}_4$ and operator $I_{\Delta t}$ is defined in \eqref{eqn:operatorflow}, we can use the BCH formula and obtain 
\begin{small}
\begin{align}\label{convergenceofhatf2}
\tilde{L}_2=& \frac{\exp\Big(\frac{(\Delta t)^2}{2}\big([\mathcal{L}_4,\mathcal{L}_3]+[\mathcal{L}_4,\mathcal{L}_2]+[\mathcal{L}_4,\mathcal{L}_1]+[\mathcal{L}_3,\mathcal{L}_2]+[\mathcal{L}_2,\mathcal{L}_1]+[\mathcal{L}_3,\mathcal{L}_1]\big)+\mathcal{O}(\Delta t)^3\Big)-I_{d}}{\Delta t}\cdot \exp (\Delta t \mathcal{L}) \nonumber\\
\to& \frac{\Delta t}{2}\big(([\mathcal{L}_4,\mathcal{L}_3]+[\mathcal{L}_4,\mathcal{L}_2]+[\mathcal{L}_4,\mathcal{L}_1]+[\mathcal{L}_3,\mathcal{L}_2]+[\mathcal{L}_2,\mathcal{L}_1]+[\mathcal{L}_3,\mathcal{L}_1]\big) +\mathcal{O}
\big((\Delta t)^2\big).
\end{align}
\end{small}
Denoting $\tilde{L}_3\equiv\tilde{L}_1+\tilde{L}_2= \frac{I_{\Delta t}-I_d}{\Delta t}$, we have  $\tilde{L}_3\to \mathcal{L}$ in $\mathcal{B}\big(\mathbb{C}_0^{2,6,\alpha}(\mathbb{T}^1\times\tilde{Y}),\mathbb{C}_0^{1,4,\alpha}(\mathbb{T}^1\times\tilde{Y})\big)$ as $\Delta t$ approaches zero. Then, applying the Prop. \ref{prop:inverseoperator}, we get,
\begin{equation}\label{convergenceofhatf}
 \lim_{\Delta t \to 0}\hat{v}_{1}=\lim_{\Delta t \to 0}\tilde{L}_3^{-1}(-v_1)=\mathcal{L}^{-1}(-v_1)=\chi_1.
\end{equation}
In addition, combining the results of the Eqns.\eqref{eqn:asymcell}, \eqref{conv:operatorstrong}, \eqref{convergenceofhatf2} and \eqref{convergenceofhatf} for the right hand side of Eq.\eqref{eqn:balancedform}, we know that when $\Delta t$ is small enough, the assertion in 
\eqref{ConvergenceResult_ftochi} is proved. The constant in the $\mathcal{O}(\Delta t)$ of \eqref{ConvergenceResult_ftochi} does not depend on the total computational time $T$, but may depend on the regularities of $v_1$, $v_2$ and the constant $\sigma$.
\end{proof}
\subsection{Convergence analysis for the effective diffusivity}\label{sec:EstimateEffectiveDiffusivity}
\noindent
This section contains the main results of our convergence analysis. We first prove that the second-order moment
of the solution obtained by using our numerical scheme has an (at most) linear growth rate. Secondly, we provide the convergence rate of our numerical method in computing the effective diffusivity.
\begin{theorem}\label{thm:boundness}
Let $\textbf{X}^n=(x^n_1,x^n_2)^{T}$ denote the solution of the two-dimensional passive tracer model \eqref{eqn:particleSDE} obtained by using our numerical scheme \eqref{scheme_analysis} with time step $\Delta t$.  If the Hamiltonian function $H(t,x_1,x_2)$ is separable, periodic and smooth (in order to guarantee the existence and uniqueness of the solution to the SDE \eqref{eqn:particleSDE}), then we can prove that the second-order moment of the solution $\textbf{X}^n$ (which can be viewed as a discrete Markov process) is at most linear growth, i.e.,
\begin{equation}\label{conj}
\max_n\big\{\mathbb{E}\frac{||\textbf{X}^n||^2}{n}\big\}\ \text{is bounded.}
\end{equation}
\begin{proof} 
We first estimate the second-order moment of the first component of $\textbf{X}^n=(x^n_1,x^n_2)^{T}$, since the other one can be estimated in the same manner. Simple calculations show that 
\begin{align}
\mathbb{E}[(x^n_1)^2|(x_1^{n-1},x_2^{n-1})]&=\mathbb{E}\big(x_1^{n-1}+v_1(t_{n-\frac{1}{2}},x_2^{n-1})\Delta t+\sigma N_1^{n-1}\big)^2\nonumber\\
&=\mathbb{E}(x_1^{n-1})^2+\Delta t \big(\sigma^2+2\mathbb{E}[x_1^{n-1}v_1(t_{n-\frac{1}{2}},x_2^{n-1})]\big)+(\Delta t)^2\mathbb{E}v_1^2(t_{n-\frac{1}{2}},x_2^{n-1}).
\end{align}
The term $\mathbb{E}[x_1^{n-1}v_1(t_{n-\frac{1}{2}},x_2^{n-1})]$ corresponds to the strength of the convection-enhanced diffusion. Our goal here is to prove that it is bounded over $n$, though it may depend on $v_1$, $v_2$ and $\sigma$. 
{\color{black}Notice that we are calculating the expectation of $(x_1^n)^2$, which is not defined in the torus space. But in the following derivation, we will show that it can be decomposed into sums of periodic functions acting on $\textbf{X}^n=(x_1^n,x_2^n)^T$. Hence after the decomposition (see Eq.\eqref{eqn:Ex1x1overn}) we can still apply the previous analysis on torus space.}

		
We now directly compute the contribution of the term $\mathbb{E}[x_1^{n-1}v_1(t_{n-\frac{1}{2}},x_2^{n-1})]$ to the effective diffusivity with the help of Eq.\eqref{DiscreteCellProblem1},  
		\begin{align}\label{ComputeEpfq}
		\Delta t\sum_{i=0}^{n-1}\mathbb{E}[x_1^iv_1(t_{i+\frac{1}{2}},x_2^i)]=\sum_{i=0}^{n-1}\mathbb{E}\big[x_1^i \big(\hat{v}_1(t_i,\textbf{X}^i)-\mathbb{E}[\hat{v}_1(t_{i+1},\textbf{X}^{i+1})|\textbf{X}^i]\big)\big].
		\end{align}
		Let $\mathcal{F}_i$ denote the filtration generated by the solution process until $\textbf{X}^i$. Notice that $x_1^i\in\mathcal{F}_i$. For the Eq.\eqref{ComputeEpfq}, we have
		\begin{align}
		\text{RHS}&=\sum_{i=0}^{n-1}\mathbb{E}\big[x_1^i \big(\hat{v}_1(t_i,\textbf{X}^i)-\hat{v}_1(t_{i+1},\textbf{X}^{i+1})\big)\big]
		\nonumber\\
		&=\sum_{i=1}^{n}\mathbb{E}\big[\hat{v}_1(t_i,\textbf{X}^i)(x_1^{i}-x_1^{i-1})\big]+\hat{v}_1(t_0,\textbf{X}^0)x_1^0-\mathbb{E}[\hat{v}_1(t_n,\textbf{X}^n)x^n_1]\nonumber\\
		&= \sum_{i=1}^{n}\mathbb{E}\big[\hat{v}_1(t_i,\textbf{X}^i)\big(v_1(t_{i-\frac{1}{2}},x_2^{i-1})\Delta t+\sigma N^{i-1}_1\big)\big]+\hat{v}_1(t_0,\textbf{X}^0)x_1^0-\mathbb{E}[\hat{v}_1(t_n,\textbf{X}^n)x^n_1].
		\end{align}
		Hence, we obtain the following result
		\begin{align}
		\frac{1}{n}\mathbb{E}\big[(x^n_1)^2|(x_1^{0},x_2^{0})\big]=&\frac{1}{n}(x_1^0)^2+\Delta t\sigma^2+2\Delta t\frac{1}{n}\sum_{i=0}^{n-1}\mathbb{E}[x_1^iv_1(t_{i+\frac{1}{2}},x_2^i)]+(\Delta t)^2\frac{1}{n}\sum_{i=0}^{n-1}\mathbb{E}v_1^2(t_{i+\frac{1}{2}},x_2^i)\nonumber\\
		=&\frac{1}{n}(x_1^0)^2+\Delta t\sigma^2+(\Delta t)^2\frac{1}{n}\sum_{i=0}^{n-1}\mathbb{E}v_1^2(t_{i+\frac{1}{2}},x_2^i)
		\nonumber\\
		&+\frac{2}{n}\sum_{i=1}^{n}\mathbb{E}\big[\hat{v}_1(t_i,\textbf{X}^i)\big(v_1(t_{i-\frac{1}{2}},x_2^{i-1})\Delta t+\sigma N^{i-1}_1\big)\big]\nonumber\\
		&+\frac{2}{n}\big(\hat{v}_1(t_0,\textbf{X}^0)x_1^0-\mathbb{E}[\hat{v}_1(t_n,\textbf{X}^n)x^n_1]\big).
		\label{eqn:Ex1x1overn}
		\end{align}
		By using the Cauchy-Schwarz inequality, we know the term  
		\begin{align}
		&\frac{2}{n}\sum_{i=1}^{n}\mathbb{E}\big[\hat{v}_1(t_i,\textbf{X}^i)\big(v_1(t_{i-\frac{1}{2}},x_2^{i-1})\Delta t+\sigma N^{i-1}_1\big)\big]\nonumber\\
		\leq&\frac{2}{n}\sum_{i=1}^{n}\mathbb{E}\big[2(\hat{v}_1(t_i,\textbf{X}^i))^2+\big((v_1(t_{i-\frac{1}{2}},x_2^{i-1})\Delta t)^2+(\sigma N^{i-1}_1)^2\big)\big]\nonumber\\
		=&\frac{2}{n}\sum_{i=1}^{n}\mathbb{E}\big[2(\hat{v}_1(t_i,\textbf{X}^i))^2+(v_1(t_{i-\frac{1}{2}},x_2^{i-1}))^2(\Delta t)^2+\sigma^2\Delta t\big]\label{eqn:var}.
		\end{align}
		 Notice that if $v_1$ and $\hat{v}_1$ are bounded in sup norm, right-hand-side of Eq.\eqref{eqn:var} is bounded for any $n$. Other terms on the right-hand side of Eq.\eqref{eqn:Ex1x1overn} are also bounded, which can be checked easily. 
		 Therefore, we can prove that $\frac{1}{n}\mathbb{E}\big[(x^n_1)^2|(x_1^{0},x_2^{0})\big]$ is bounded.  Repeat the same trick, we know that $\frac{1}{n}\mathbb{E}\big[(x^n_2)^2|(x_1^{0},x_2^{0})\big]$ is also bounded. Thus, the assertion in Eq.\eqref{conj} is proved. 
	\end{proof}
\end{theorem}		
In practice, we shall first choose a time step $\Delta t$ and run our numerical scheme \eqref{scheme} 
to compute the effective diffusivity until the result converges to a constant, which may depend on $\Delta t$.  As such, we shall prove that the limit of the constant converges to the exact effective diffusivity of the original passive tracer model as $\Delta t$ approaches zero. Namely, we shall prove that our numerical scheme is robust in computing effective diffusivity. More details on the numerical results will be given in Section \ref{sec:numerical_result}. 
\begin{theorem}\label{thm:convergence}
	Let $x^n_1$, $n=0,1,....$ be the first component of the numerical solution obtained by using the scheme \eqref{scheme} and $\Delta t$ denote the time step. We have the convergence estimate of the effective diffusivity as 
	\begin{align}\label{est:order}
	\lim_{n\to\infty}\frac{\mathbb{E}(x^n_1)^2}{n\Delta t}= \sigma^2+2\int_{\mathbb{T}^2} \chi_1v_1+\mathcal{O}(\Delta t),
	\end{align}
	where the constant in $\mathcal{O}(\Delta t)$  may depend on the regularity of $v_1$, $v_2$ and the constant $\sigma$, but does not depend on the computational time $T$. 
	\begin{proof}
		We will prove the statement by direct computation. We divide both sides of the Eq.\eqref{eqn:Ex1x1overn} by $\Delta t$ and obtain 
		\begin{align}\label{eqn:mainterm}
		\frac{1}{n\Delta t}\mathbb{E}\big[(x^n_1)^2|(x_1^{0},x_2^{0})\big]=&\frac{1}{n\Delta t}(x_1^0)^2+\sigma^2+ \frac{\Delta t}{n}\sum_{i=0}^{n-1}\mathbb{E}v_1^2(t_{i+\frac{1}{2}},x_2^i)\nonumber\\
		&+\frac{2}{n\Delta t}\sum_{i=1}^{n}\mathbb{E}\big[\hat{v}_1(t_i,\textbf{X}^i)\big(v_1(t_{i-\frac{1}{2}},x_2^{i-1})\Delta t+\sigma N^{i-1}_1\big)\big]\nonumber\\
		&+\frac{2}{n\Delta t}\big(\hat{v}_1(t_0,\textbf{X}^0)x_1^0-\mathbb{E}[\hat{v}_1(t_n,\textbf{X}^n)x^n_1]\big).
		\end{align}
		First, we notice that for a fixed $\Delta t$, the terms $\frac{1}{n\Delta t}(x_1^0)^2$ and 
		$\frac{2}{n\Delta t}\hat{v}_1(t_0,\textbf{X}^0)x_1^0$ converge to zero as $n\to\infty$, where we have used the fact that
		$\hat{v}_1(t_0,\textbf{X}^0)$ is bounded. Also notice that the term $\frac{\Delta t}{n}\sum_{i=0}^{n-1}\mathbb{E}v_1^2(t_{i+\frac{1}{2}},x_2^i)$ is $\mathcal{O}(\Delta t)$, due to the term $(v_1)^2$ is bounded. Then, for a fixed $\Delta t$, we have 
		\begin{equation}
		\lim_{n\to\infty}\frac{2}{n\Delta t}\big|\mathbb{E}[\hat{v}_1(\textbf{X}^n)x^n_1]\big|\leq \lim_{n\to\infty}\frac{2}{\sqrt{n}\Delta t}||\hat{v}_1||_{\infty}\mathbb{E}|\frac{x^n_1}{\sqrt{n}}|\leq \lim_{n\to\infty}\frac{1}{\sqrt{n}\Delta t}||\hat{v}_1||_{\infty}\mathbb{E}\big[\frac{(x^n_1)^2}{n}+1\big]=0,
		\end{equation}
		where the term $\mathbb{E}[\frac{(x^n_1)^2}{n}]$ is bounded due to the Theorem \ref{thm:boundness} and 
		$||\hat{v}_1||_\infty\to||\chi_1||_\infty<\infty $ due to the Theorem \ref{thm:disc-to-cont}.
		
		Therefore, we only need to focus on the estimate of terms in the second line of Eq.\eqref{eqn:mainterm}, which corresponds to the convection-enhanced diffusion effect. 
		Notice that $\hat{v}_1\in\mathbb{C}^{2,6,\alpha}$, we compute the Ito-Taylor series approximation of $\hat{v}_1(t_{i},\textbf{X}^i)$,
		\begin{align}\label{eqn:taylor}
		\hat{v}_1(t_{i},\textbf{X}^i)&=\hat{v}_1(t_{i-1},\textbf{X}^{i-1})+\hat{v}_{1,x_1}(t_{i-1},\textbf{X}^{i-1})\big(v_1(t_{i-\frac{1}{2}},x_2^{i-1})\Delta t+\sigma N^{i-1}_1\big)\nonumber\\
		&+\hat{v}_{1,x_2}(t_{i-1},\textbf{X}^{i-1})\big(v_2\big(t_{i-\frac{1}{2}},x_1^{i-1}\big)\Delta t+\sigma N^{i-1}_2\big)\nonumber\\
		&+\frac{1}{2}\big(\hat{v}_{1,x_1x_1}(t_{i-1},\textbf{X}^{i-1})+\hat{v}_{1,x_2x_2}(t_{i-1},\textbf{X}^{i-1})\big)\sigma^2\Delta t+\mathcal{O}\big((\Delta t)^2\big),
		\end{align}		
		where we have used the fact that $v_2\big(t_{i-\frac{1}{2}},x_1^{i-1}+v_1(t_{i-\frac{1}{2}},x_2^{i-1})\Delta t\big)=v_2\big(t_{i-\frac{1}{2}},x_1^{i-1}\big)+\mathcal{O}(\Delta t)$, when $\Delta t$ is small and $v_2$ is smooth. 
		 Since $\hat{v}_1\to\chi_1$ in $\mathbb{C}_0^{2,6,\alpha}$, the truncated term $\mathcal{O}\big((\Delta t)^2\big)$ in Eq.\eqref{eqn:taylor} is uniformly bounded when $\Delta t$ is small enough. Substituting the Taylor expansion of $\hat{v}_1(t_{i},\textbf{X}^i)$ in Eq.\eqref{eqn:taylor} into the target term of our estimate (i.e., terms in the second line of Eq.\eqref{eqn:mainterm}), we get 
		\begin{align}
		\mathbb{E}\big[\hat{v}_1(t_{i},\textbf{X}^i)&(v_1(t_{i-\frac{1}{2}},x_2^{i-1})\Delta t+\sigma N^{i-1}_1)\big]=\mathbb{E}\Big[\Big(v_1(t_{i-\frac{1}{2}},x_2^{i-1})\Delta t+\sigma N^{i-1}_1\Big)\cdot\nonumber\\
		&\Big(\hat{v}_1(t_{i-1},\textbf{X}^{i-1})+\hat{v}_{1,x_1}(t_{i-1},\textbf{X}^{i-1})
		\big(v_1(t_{i-\frac{1}{2}},x_2^{i-1})\Delta t+\sigma N^{i-1}_1\big)\nonumber\\
		&+\hat{v}_{1,x_2}(t_{i-1},\textbf{X}^{i-1})\big(v_2(t_{i-\frac{1}{2}},x_1^{i-1})\Delta t+\sigma N^{i-1}_2\big)\nonumber\\
		&+\frac{1}{2}\big(\hat{v}_{1,x_1x_1}(t_{i-1},\textbf{X}^{i-1})+\hat{v}_{1,x_2x_2}(t_{i-1},\textbf{X}^{i-1})\big)\sigma^2\Delta t+\mathcal{O}((\Delta t)^2)\Big)\Big].
		\end{align}
		Combining the terms with the same order of $\Delta t$, we obtain 
		\begin{align}
		&\mathbb{E}\big[\hat{v}_1(t_{i},\textbf{X}^i)\big(v_1(t_{i-\frac{1}{2}},x_2^{i-1})\Delta t+\sigma N^{i-1}_1\big)\big] \nonumber \\
		=&\Delta t\mathbb{E}\big[\hat{v}_1(t_{i-1},\textbf{X}^{i-1})v_1(t_{i-\frac{1}{2}},x_2^{i-1})+\sigma^2\hat{v}_{1,x_1}(t_{i-1},\textbf{X}^{i-1})\big]
		+\mathcal{O}((\Delta t)^2),
		\end{align}
where we have used the facts that: (1) $\textbf{X}^{i-1}$ is independent of $N^{i-1}_1$ and $N^{i-1}_2$ 
so the expectations of the corresponding terms vanish; (2) $N^{i-1}_1$ and $N^{i-1}_2$ are independent
so $\mathbb{E}(N^{i-1}_1N^{i-1}_2)=0$; and (3) $\mathbb{E}(N^{i-1}_1)^2=\Delta t$. 

Finally, by using the Theorem \ref{thm:main-per} and noticing the invariant measure is the uniform measure, we obtain from Eq.\eqref{eqn:mainterm} that
		\begin{align}
		\lim_{n\to\infty} \frac{1}{n\Delta t}\mathbb{E}\big[(x^n_1)^2|(x_1^{0},x_2^{0})\big]=\sigma^2+2\int (\hat{v_1}v_1+\sigma^2\hat{v}_{1,x_1})+\mathcal{O}(\Delta t).
		\end{align}
		Thus, our statement in the Eq.\eqref{est:order} is proved 
		using the facts that $\hat{v}_1$ converges to $\chi_1$ (see Theorem \ref{thm:disc-to-cont})
		and $\int \hat{v}_{1,x_1}=0$.
	\end{proof}
\end{theorem}
\begin{remark}
If we divide two on both sides of the Eq.\eqref{est:order}, we can find that our result 
recovers the definition of the effective diffusivity $D^{E}_{11}$ defined in the Eq.\eqref{Def_EffectiveDiffusivity_Euler}.	Recall that $D_0=\sigma^2/2$. 
Theorem \ref{thm:convergence} reveals the connection of the definition of the effective diffusivity using the Eulerian framework and Lagrangian framework; see Eq.\eqref{Def_EffectiveDiffusivity_Euler} and Eq.\eqref{Def_EffectiveDiffusivity_Lagrangian}, which is fundamental in this context. For 3D time-dependent flow problems, the Eulerian framework has good theoretical values but the Lagrangian framework is computationally accessible.
\end{remark}
\begin{remark}
For the second component of the numerical solution, i.e., $x^n_2$, $n=0,1,...$, we can obtain the similar convergence result in computing the effective diffusivity.  First we consider $\tilde{v}_2(t,x_1,x_2):=v_2(t,x_1+v_1(t,x_2)\Delta t)$ and notice that $\tilde{v}_2-v_2=\mathcal{O}(\Delta t)$ and $\int_{\mathbb{T}}\tilde{v}_2\mathrm{d}x_1=0$. The remaining part of the proof is essentially the same as the results obtained in Sections \ref{sec:ProbabilisticProof}, \ref{sec:ConvergenceDiscreteProblem} and \ref{sec:EstimateEffectiveDiffusivity}, so we skip the details here. 
\end{remark}
	
\subsection{Generalizations to high-dimensional cases}\label{sec:HighDimensionalCases}  
\noindent
To show the essential idea of our probabilistic approach in proving the convergence rate of the numerical schemes, we have carried out our convergence analysis based on a two-dimensional model problem \eqref{eqn:particleSDE}. In fact, the extension of our approach to higher-dimensional problems is straightforward. Now we consider a high-dimensional problem as follow, 
\begin{equation}
\label{eqn:multiDSDE}
\mathrm{d}\textbf{X} =\textbf{v}(t,\textbf{X} )\mathrm{d}t+\Sigma \mathrm{d}\textbf{w}(t),
\end{equation}
where $\textbf{X}=(x_1,x_2,\cdots,x_d)^{T}\in \mathbb{R}^{d}$  is the position of a particle, 
$\textbf{v}=(v_1,v_2,\cdots,v_d)^{T}\in \mathbb{R}^{d}$ is the Eulerian velocity field at position $\textbf{X}$, $\Sigma$ is a $d\times d$ constant non-singular matrix, and $\textbf{w}(t)$ is a $d$-dimension Brownian motion vector. In particular, we assume the component $v_i$ does not depend on $x_i$, $i=1,...,d$.  Thus, the incompressible condition for $\textbf{v}(t,\textbf{X})$ (i.e. $\nabla_\textbf{X}\cdot \textbf{v}(t,\textbf{X})=0$) is easily guaranteed. 


For a deterministic and divergence-free dynamical system, Feng et. al. proposed a volume-preserving method \cite{KangShang1995volume}, which splits a $d$-dimensional problem into $d-1$ subproblems with each of them being a two-dimensional problem and thus being volume-preserving. We shall modify Feng's method (first-order case) by including the randomness as the last subproblem to take into account the additive noise, i.e., 
\begin{equation}\label{scheme:multiDimension}
	{\color{black}\begin{cases}
			x_1^{*}=x_1^{n-1}+\Delta t v_1(t_n+\frac{\Delta t}{2},x_2^{n-1},x_3^{n-1},x_4^{n-1},\cdots,x_{d-1}^{n-1}, x_d^{n-1}),\\
			x_2^{*}=x_2^{n-1}+\Delta t v_2(t_n+\frac{\Delta t}{2},x_1^{*},x_3^{n-1},x_4^{n-1},\cdots,x_{d-1}^{n-1}, x_d^{n-1}),\\
			x_3^{*}=x_3^{n-1}+\Delta t v_3(t_n+\frac{\Delta t}{2},x_1^{*},x_2^{*},x_4^{n-1},\cdots,x_{d-1}^{n-1}, x_d^{n-1}),\\
			\cdots,\\ 
			x_d^{*}=x_d^{n-1}+\Delta t v_d(t_n+\frac{\Delta t}{2},x_1^{*},x_2^{*},x_3^{*},x_4^{*},\cdots, x_{d-1}^{*}),\\
			\textbf{X}^{n}=\textbf{X}^{*}+\Sigma(\textbf{W}^{n}-\textbf{W}^{n-1}),
	\end{cases}}
\end{equation}
where $\textbf{X}^{*}=(x_1^{*},x_2^{*},\cdots,x_d^{*})^{T}$, $\textbf{W}^{n}-\textbf{W}^{n-1}$ is a $d$-dimensional independent random vector with each component of the form $\sqrt{\Delta t}\xi_{i}$, $\xi_{i}\sim\mathcal{N}(0,1)$, and $\textbf{X}^{n}=(x_1^{n},x_2^{n},\cdot\cdot\cdot,x_d^{n})^T$  is the numerical approximation to the exact solution $\textbf{X}(t_n)$ to the SDE \eqref{eqn:multiDSDE}  at time $t_n=n\Delta t$. 

The techniques of the convergence analysis for the two-dimensional problem can be applied to 
high-dimensional problems without much difficulty. For the high-dimensional problem \eqref{eqn:multiDSDE}, 
the smoothness and  strict positivity of the transition kernel in the discrete process can be
guaranteed if one assumes that the covariance matrix $\Sigma$ is non-singular and the scheme 
\eqref{scheme:multiDimension} is explicit. According to our assumption for the velocity field, the scheme \eqref{scheme:multiDimension} is volume-preserving for each step. Thus, the solution to the first-order modified equation is divergence-free and the invariant measure on the torus (defined by $\mathbb{R}^d/\mathbb{Z}^d$, when the period is $1$) remains uniform for all $t$. Finally, the convergence of the cell problem can be studied by using the BCH formula \eqref{eqn:BCHformula} with $d+2$ differential operators. Recall that in the Eq.\eqref{eqn:operatorflow} we have four differential operators when we study the two-dimensional problem. 

Therefore, our numerical methods are robust in computing effective diffusivity for high-dimensional problems, which will be demonstrated through time-dependent chaotic flow problems in three-dimensional space in Section \ref{sec:numerical_result}.

\section{Numerical results}\label{sec:numerical_result}
\noindent 
In this section, we will present numerical examples to verify the convergence analysis 
of the proposed method in computing effective diffusivity for time-dependent chaotic flows. In addition, 
we will investigate the convection-enhanced diffusion phenomenon in 3D time-dependent flow, i.e., the time-dependent ABC flow and the time-dependent Kolmogorov flow. Without loss of generality, we compute the quantity $\frac{\mathbb{E}(x_1(T))^2}{2T}$, which is used to approximate $D^{E}_{11}$ in the effective diffusivity matrix \eqref{Def_EffectiveDiffusivity_Euler}. 

\subsection{Verification of the convergence rate}\label{sec:ConvergenceRate}
\noindent
We first consider a two-dimensional passive tracer model. 
Let $(x_1,x_2)^{T}\in \mathbb{R}^{2}$ denote the position of a particle. Its motion is described by the following SDE,  
\begin{equation}\label{eqn:TwoD-TDFlow}
\begin{cases}
\mathrm{d}x_1=\sin\big(4 x_2+1+\sin(2\pi t)\big)\exp\big(\cos\big(4 x_2 +1+\sin(2\pi t)\big)\big)\mathrm{d}t+\sigma \mathrm{d}w_{1,t},\\
\mathrm{d}x_2= \cos\big(2 x_1+\sin(2\pi t)\big)\exp\big(\sin\big(2 x_1 +\sin(2\pi t)\big)\big)\mathrm{d}t+\sigma \mathrm{d}w_{2,t},
\end{cases}
\end{equation}
where $\sigma=\sqrt{2\times 0.1}$, $w_{i,t}$, $i=1,2$ are independent Brownian motions, and the initial data $(x_1^0,x_2^0)^{T}$ follows uniform distributions in $[-0.5,0.5]^2$. One can easily verify the velocity field in \eqref{eqn:TwoD-TDFlow} is time-dependent and divergence-free.

In our numerical experiments, we use Monte Carlo samples to discretize the Brownian motions $w_{1,t}$ and $w_{2,t}$. The sample number is denoted by $N_{mc}$. We choose $\Delta t_{ref}=\frac{1}{2^{12}}$ and $N_{mc}=3,200,000$ to solve the SDE \eqref{eqn:TwoD-TDFlow} to compute the reference solution, i.e., the ``exact'' effective diffusivity, where the final computational time is $T=3000$ so that the calculated effective diffusivity converges to a constant. In fact, we find that the passive tracer model will enter a mixing stage if the computational time is bigger than $T=1000$.  It takes about $17$ hours to compute the reference solution on a 80-core server (HPC2015 System at HKU). The reference solution for the effective diffusivity is $D^{E}_{11}=0.219$.

In Fig.\ref{fig:convergence-2D-flow}, we plot the convergence results of the effective diffusivity using our method (i.e., $\frac{\mathbb{E}(x_1(T))^2}{2T}$) with respective to different time-step $\Delta t$ at  $T=3000$. In addition, we show a fitted straight line with the slope $1.04$, i.e., the convergence rate is about $(\Delta t)^{1.04}$. This numerical result verifies the convergence analysis in Theorem \ref{thm:convergence}. 


\begin{figure}[tbph]
	\centering
	\subfigure[]{
		\includegraphics[width=0.45\linewidth]{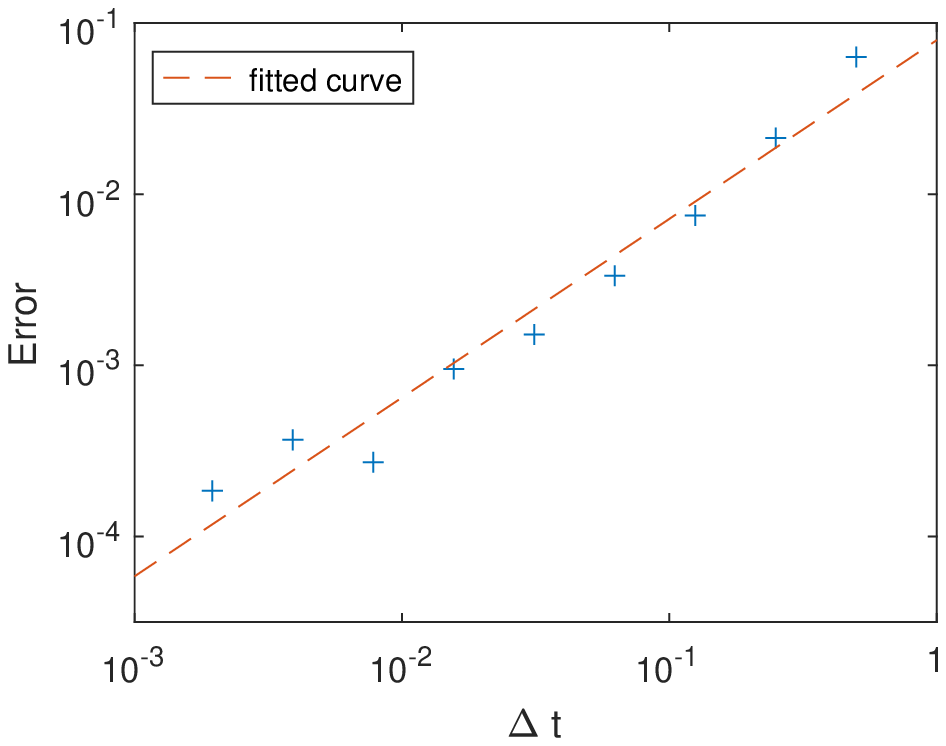}
		\label{fig:convergence-2D-flow}
	}%
	\subfigure[]{
		\includegraphics[width=0.45\linewidth]{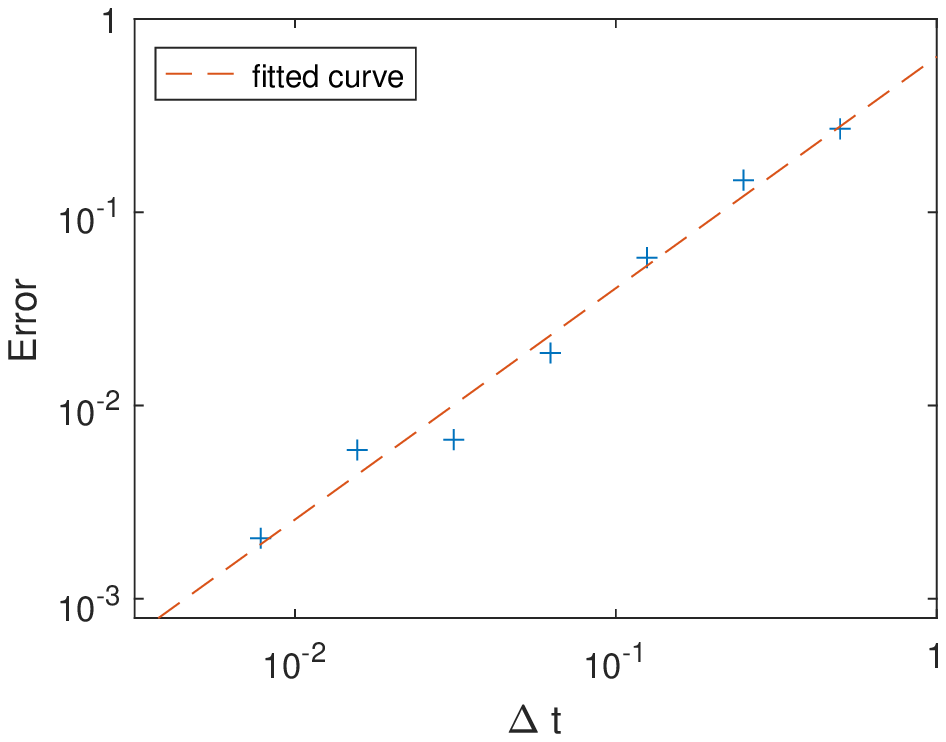}
		\label{fig:convergenced1e-4eps1e-1}
	} 
	\caption{Error of $D^{E}_{11}$ for two time-dependent flows with different time-steps. 
	(a) 2D time-dependent chaotic flow, fitted slope $\approx 1.04$; (b) 3D time-dependent Kolmogorov flow,  fitted slope $\approx 1.22$.}	
	\label{fig:convergence-2D-3D-flow}
\end{figure}


To further study the accuracy and robustness of our method for long-time integration, we consider a 3D time-dependent Kolmogorov flow problem. Let $(x_1,x_2,x_3)^{T}\in \mathbb{R}^{3}$ denote the position of a particle.  
The motion of a particle moving in the 3D time-dependent Kolmogorov flow is described by the following SDE, 
\begin{equation}\label{TD-Kflow-Eq}
\begin{cases}
\mathrm{d}x_1=\sin\big(x_3+\epsilon\sin(2\pi t)\big)\mathrm{d}t +\sigma \mathrm{d}w_{1,t},\\
\mathrm{d}x_2=\sin\big(x_1+\epsilon\sin(2\pi t)\big)\mathrm{d}t +\sigma \mathrm{d}w_{2,t},\\
\mathrm{d}x_3=\sin\big(x_2+\epsilon\sin(2\pi t)\big)\mathrm{d}t +\sigma \mathrm{d}w_{3,t}.
\end{cases} 
\end{equation} 
where $w_{1,t}$, $w_{2,t}$ and $w_{3,t}$ are independent Brownian motions. When $\epsilon=0$, the velocity field in \eqref{TD-Kflow-Eq} corresponds to the Kolmogorov flow \cite{KflowGalloway:1992}. The Kolmogorov flow possesses very chaotic behaviors \cite{childress1995stretch}, which brings challenges for our method.  

In our numerical experiment, we choose $\epsilon=10^{-1}$ and $\sigma=\sqrt{2\times 10^{-3}}$ in the Eq.\eqref{TD-Kflow-Eq}. We choose $\Delta t_{ref}=\frac{1}{2048}$ and $N_{mc}=480,000$ to 
compute the reference solution for the SDE \eqref{TD-Kflow-Eq}, i.e., the ``exact'' effective diffusivity.  In our numerical tests, we find that the passive tracer model will enter a mixing stage if the computational time is bigger than $T=2000$. To show the accuracy and robustness of our numerical scheme, we set $T=10^5$ here. It takes about 59 hours to compute the reference solution on the server and the reference solution for the effective diffusivity is $D^{E}_{11}=0.693$.  

In Fig.\ref{fig:convergenced1e-4eps1e-1}, we plot the convergence results of the effective diffusivity using our method with respect to different time-step $\Delta t$. In addition, we show a fitted straight line with the slope $1.22$, i.e., the convergence rate is about $(\Delta t)^{1.22}$. This numerical result again agrees with our error analysis. 


\subsection{Investigation of the convection-enhanced diffusion phenomenon}\label{sec:convection-enhanced}
\noindent 
As we have already demonstrated in Section \ref{sec:ConvergenceRate}, our method is very accurate and robust for long-time integration. Here, we will study the dependence of the effective diffusivity $D^E_{11}$ on different parameters in the time-dependent flows. First of all, we solve Eq.\eqref{TD-Kflow-Eq} and carry out the test for the 3D time-dependent Kolmogorov flow. 
 
In Fig.\ref{fig:tdkflowVaryD0VaryEps}, we show the time evolution of $\frac{\mathbb{E}(x_1(t))^2}{2t}$ for 
different $D_0$'s (here $D_0=\sigma^2/2$) and for four different $\epsilon$'s, where the result in Fig.\ref{fig:kflow_result} corresponding to the time-independent Kolmogorov flow (see Figure 6 of \cite{Zhongjian2018sharp}).  Notice that in Eq.\eqref{TD-Kflow-Eq} the parameter $\epsilon$ controls the strength of the time dependence. For each $D_0$ and $\epsilon$, we use $N_{mc}=240,000$ particles to solve the SDE \eqref{TD-Kflow-Eq}. We find that for each given $D_0$, the time evolution of $\frac{E(x_1(t))^2}{2t}$ converges as $\epsilon$ approaches zero. This can be rigorously justified through analysis; see \ref{sec:Limit-K-Flow}. In addition, we observe a certain amount of enhanced diffusion when $D_0$ decreases. However, the dependency of $D^E_{11}$ on $D_0$ is quite different from the pattern of the time-dependent ABC flow, which is known as the maximal enhancement and will be discussed later; see Fig.\ref{fig:tdabcflowresult}. 

To study the dependence of $D^E_{11}$ on $D_0$ and $\epsilon$, we choose different $\epsilon$'s and $D_0$'s and compute the corresponding effective diffusivity $D^E_{11}$. In this experiment, we use $\Delta t=2^{-7}$ and $N_{mc}=240,000$ particles to compute. The final computational time is $T=10^5$ so that the particles are fully mixed. We show the numerical results in Fig.\ref{fig:tdkflowVaryD0OneEps}.

We find that for each given $D_0$ as $\epsilon$ decreases the corresponding effective diffusivity $D^E_{11}$ converges to the effective diffusivity $D^E_{11}$ associated with $\epsilon=0$. This means the time dependency of $\epsilon$ improves the chaotic property of Kolmorogov flow though, it does not change the pattern of convection-enhanced diffusion in the Kolmogorov flow. When $\epsilon \leq 1$ the fitted slope within $D_0\in[10^{-5},10^{-3}]$ is $-0.2$, which indicates that $D^E_{11}\sim \mathcal{O}(1/D_0^{0.2})$. 
We call this behavior a sub-maximal enhancement, which may be explained by the fact that the Kolmogorov flow is more chaotic than the ABC flow \cite{KflowGalloway:1992}. The chaotic trajectories in Kolmogorov flow enhance diffusion much less than channel-like structures such as the ballistic orbits of ABC flows \cite{mcmillen2016ballistic,xin2016periodic}.   

\begin{figure}[tbph]
	\centering
	\subfigure[]{
		\includegraphics[width=0.45\linewidth]{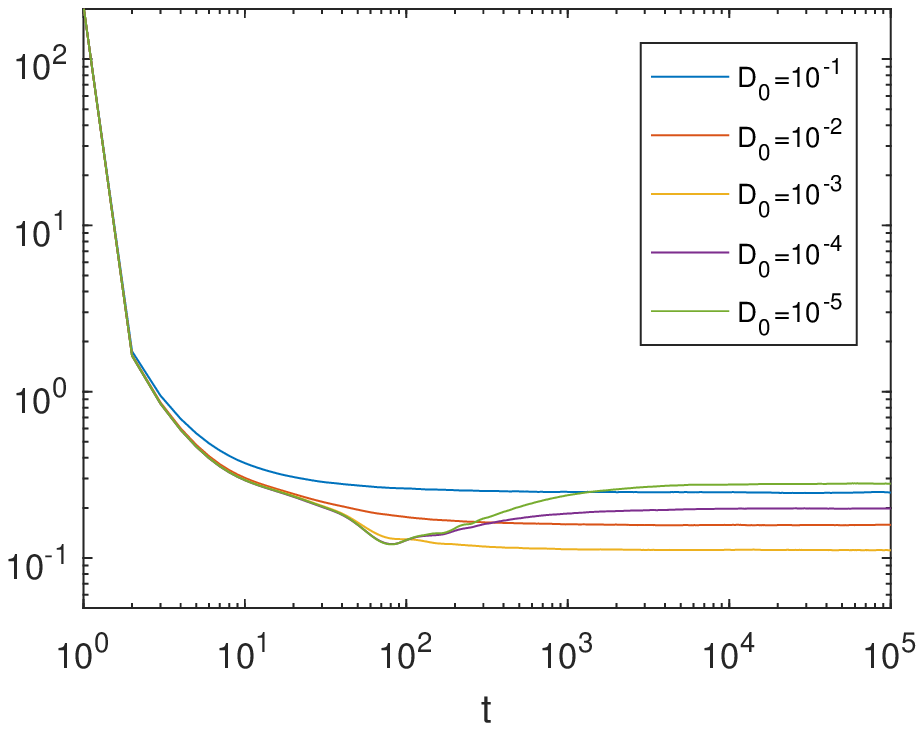}
		\label{fig:epslimit_eps1e1}
	}%
	\subfigure[]{
		\includegraphics[width=0.45\linewidth]{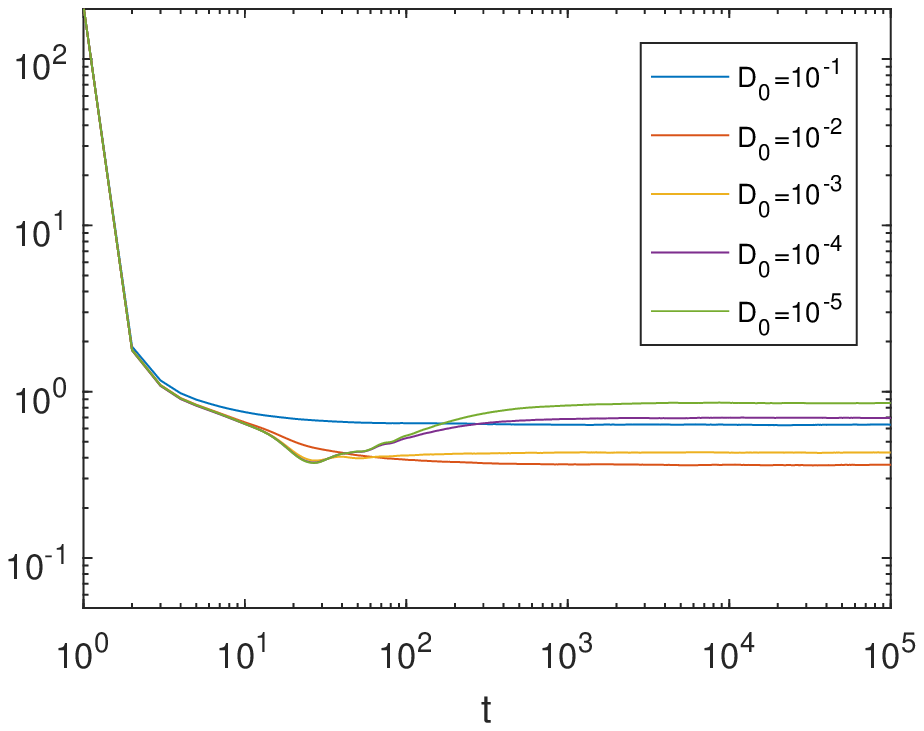}
		\label{fig:epslimit_eps1e0}
	}\\
	\subfigure[]{
		\includegraphics[width=0.45\linewidth]{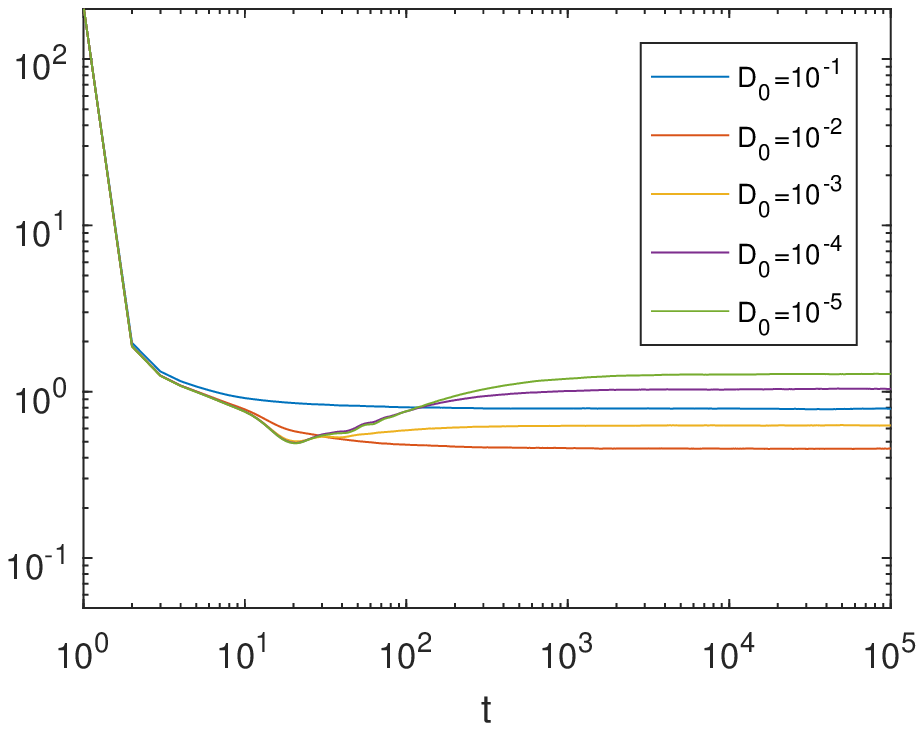}
		\label{fig:epslimit_eps1e-1}
	}%
	\subfigure[]{
		\includegraphics[width=0.45\linewidth]{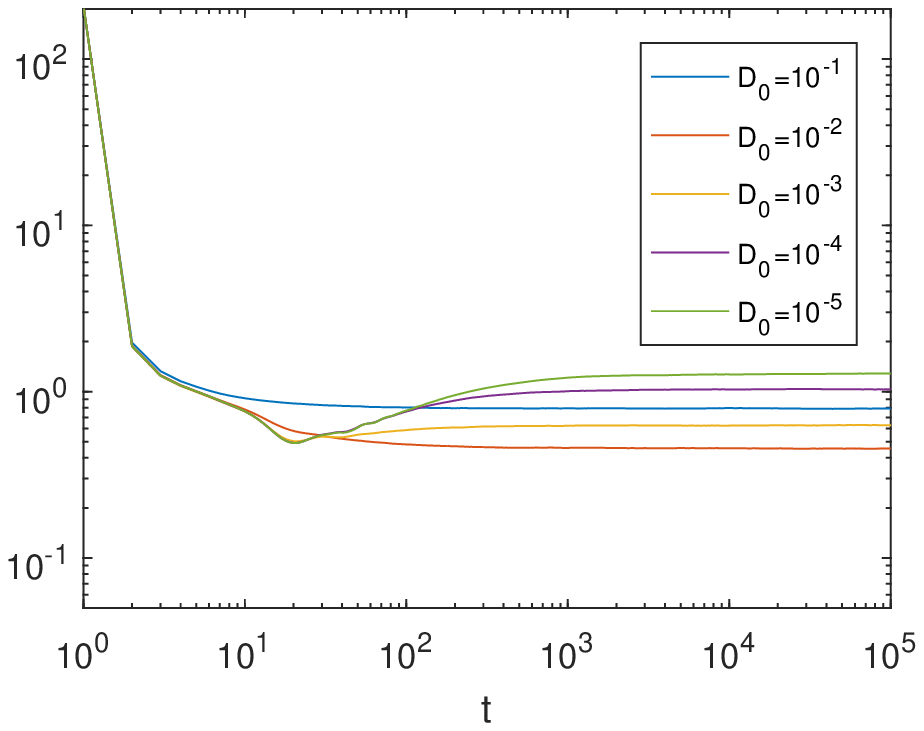}
		\label{fig:kflow_result}
	}
	\caption{Time evolution of the $\frac{\mathbb{E}(x_1(t))^2}{2t}$ for different $D_0$'s and $\epsilon$'s.
		(a) $\epsilon=10$, (b) $\epsilon=1$, (c) $\epsilon=0.1$, (d) $\epsilon=0$. }	
	\label{fig:tdkflowVaryD0VaryEps}
\end{figure}

\begin{figure}[tbph]
	\centering
	\includegraphics[width=0.5\linewidth]{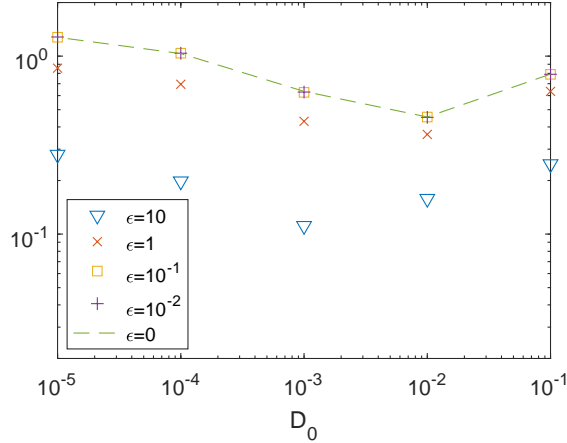}
	\caption{Convection-enhanced diffusion with a sub-maximal enhancement in the time-dependent Kolmogorov flow.}
	\label{fig:tdkflowVaryD0OneEps}
\end{figure}

Next, we use our stochastic structure-preserving scheme to solve time-dependent ABC flow problems. Let $(x_1,x_2,x_3)^{T}\in \mathbb{R}^{3}$ denote the position of a particle in the 3D Cartesian coordinate system. The motion of a particle moving in the 3D time-dependent ABC flow is described by the following SDE, 
\begin{equation}\label{TD-ABCflow-Eq}
\begin{cases}
\mathrm{d}x_1=A\sin\big(x_3+\epsilon\sin(2\pi t)\big)\mathrm{d}t+C\cos\big(x_2+\epsilon\sin(2\pi t)\big)\mathrm{d}t +\sigma \mathrm{d}w_{1,t},\\
\mathrm{d}x_2=B\sin\big(x_1+\epsilon\sin(2\pi t)\big)\mathrm{d}t+A\cos\big(x_3+\epsilon\sin(2\pi t)\big)\mathrm{d}t +\sigma \mathrm{d}w_{2,t},\\
\mathrm{d}x_3=C\sin\big(x_2+\epsilon\sin(2\pi t)\big)\mathrm{d}t+B\cos\big(x_1+\epsilon\sin(2\pi t)\big)\mathrm{d}t +\sigma \mathrm{d}w_{3,t},
\end{cases}
\end{equation}
where $w_{1,t}$, $w_{2,t}$ and $w_{3,t}$ are independent Brownian motions. For $\epsilon=0$ and $\sigma=0$, the velocity field in \eqref{TD-ABCflow-Eq} corresponds to the standard ABC flow \cite{dombre1986chaotic}. The ABC flow is a three-dimensional incompressible velocity field which is an exact solution to the Euler's equation. It is notable as a simple example of a fluid flow that can have chaotic trajectories. 
In our numerical experiments, we set $A=B=C=1$.

In Fig.\ref{fig:tdABCflowVaryD0VaryEps}, we show the time evolution of the $\frac{\mathbb{E}(x_1(t))^2}{2t}$ for 
different $D_0$'s (here $D_0=\sigma^2/2$) and for four different $\epsilon$'s, where 
the result in Fig.\ref{fig:abc_result} corresponding to the time-independent ABC flow (see Figure 3 of \cite{Zhongjian2018sharp}). Again the parameter $\epsilon$ controls the strength of the time dependence.  For each $D_0$ and $\epsilon$, we use $N_{mc}=240,000$ particles to solve the SDE \eqref{TD-ABCflow-Eq}. We find that for each given $D_0$, the time evolution of the $\frac{\mathbb{E}(x_1(t))^2}{2t}$ converges when $\epsilon$ converges to zero.  However, we observe two different patterns compared with the results shown in Fig.\ref{fig:tdkflowVaryD0VaryEps}. First, when 
we decrease $D_0$, it takes a longer time for the system to enter a mixing stage. Second,  we observe a large amount of enhanced diffusion when $D_0$ decreases.

To further investigate the dependence of $D^E_{11}$ on $D_0$ and $\epsilon$, we choose different $\epsilon$'s and $D_0$'s and compute the corresponding effective diffusivity $D^E_{11}$. In this experiment, we use $\Delta t=2^{-7}$ and $N_{mc}=240,000$ particles to compute. The final computational time is $T=10^5$ so that the particles are fully mixed. 

In Fig.\ref{fig:tdabcflowresult}, we show the numerical results. We find that for each given $D_0$, as $\epsilon$ decreases the corresponding effective diffusivity $D^E_{11}$ converges to the effective diffusivity $D^E_{11}$ associated with $\epsilon=0$. Thus, the time-dependent ABC flow has a similar convection-enhanced diffusion behavior as the time-independent ABC flow. The fitted slope within $D_0\in[10^{-5},10^{-1}]$ is about $-1.0$, which indicates that $D^E_{11}\sim \mathcal{O}(1/D_0^{1})$. 
This result indicates that the $D^E_{11}$ of the time-dependent ABC flow achieves the upper-bound of Eq.\eqref{eqn:maximaldiffusion}, i.e. the maximal enhancement. This maximal enhancement phenomenon may be attributed to the ballistic orbits of the ABC flow, where the time-independent case was discussed in \cite{mcmillen2016ballistic,xin2016periodic}. 

Moreover, our result for $D_0\in[10^{-3},10^{-1}]$ and $\epsilon=0$ recovers the same phenomenon as the Fig.2 in \cite{Biferale:95}, which was obtained by using the Eulerian framework, i.e., solving a cell problem. In Fig.\ref{fig:tdabcflowresult}, our method can be easily used to compute the effective diffusivity when $D_0\in[10^{-5},10^{-4}]$. It will be, however, extremely expensive for the Eulerian framework since one needs to solve a convection-dominated PDE \eqref{CellProblem_EffectiveDiffusivity} in 3D space, whose P\'{e}clet number is proportion to $\frac{1}{D_0}$.
\begin{figure}[tbph]
	\centering
	\subfigure[]{
		\includegraphics[width=0.45\linewidth]{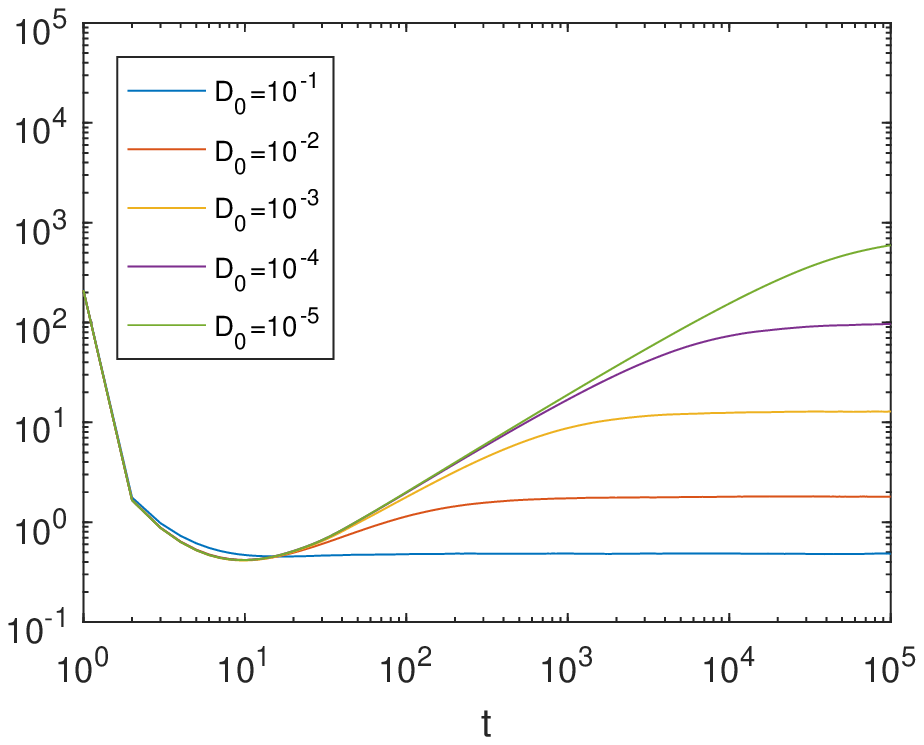}
		\label{fig:abc_epslimit_eps1e1}
	}%
	\subfigure[]{
		\includegraphics[width=0.45\linewidth]{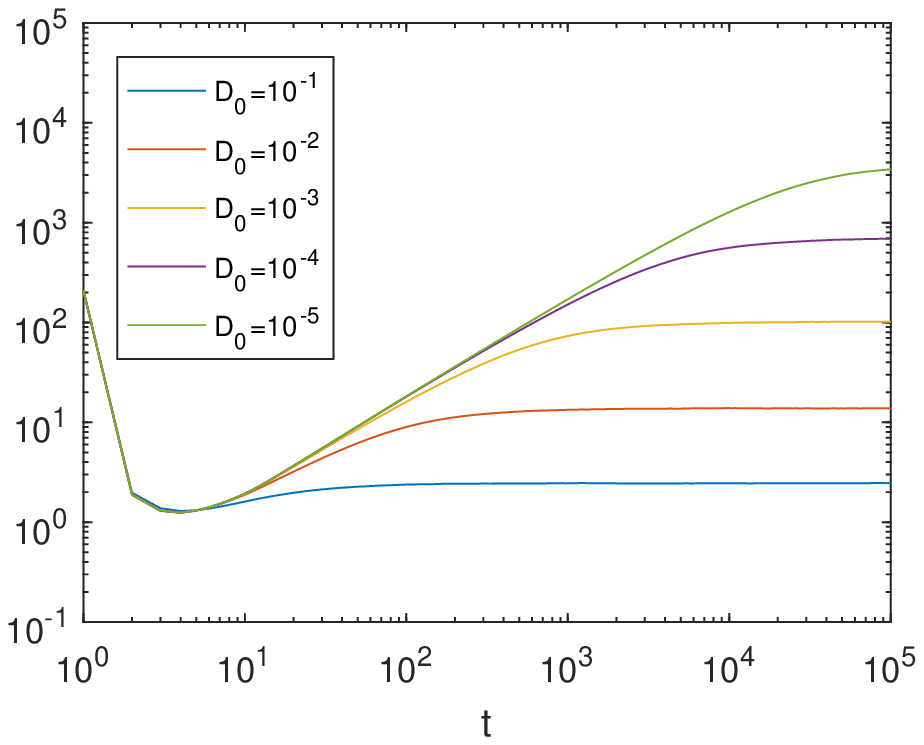}
		\label{fig:abc_epslimit_eps1e0}
	}\\
	\subfigure[]{
		\includegraphics[width=0.45\linewidth]{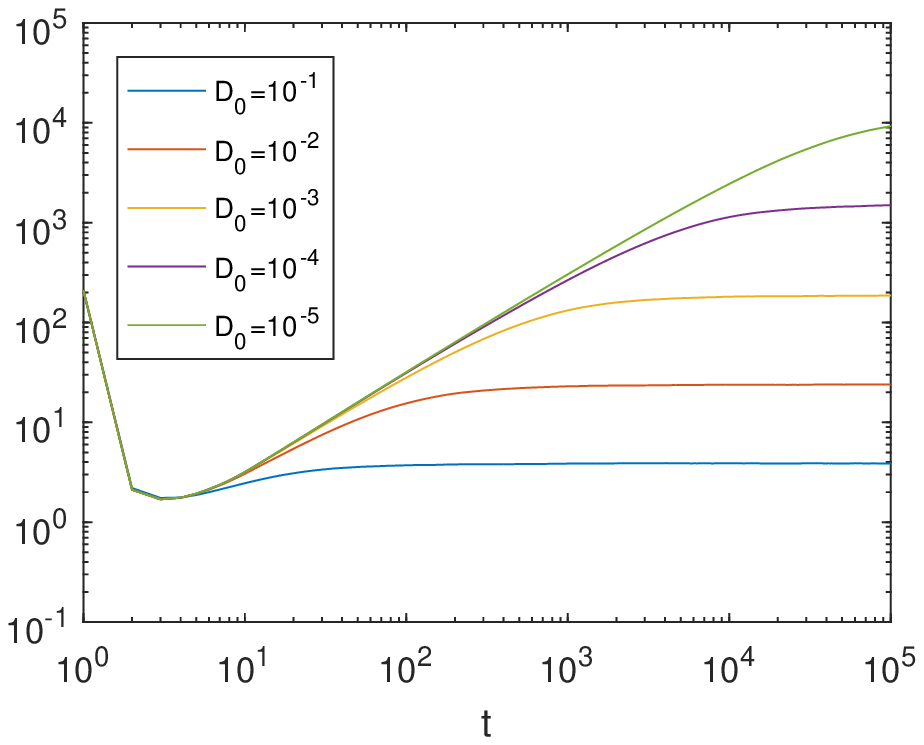}
		\label{fig:abc_epslimit_eps1e-1}
	}%
	\subfigure[]{
		\includegraphics[width=0.45\linewidth]{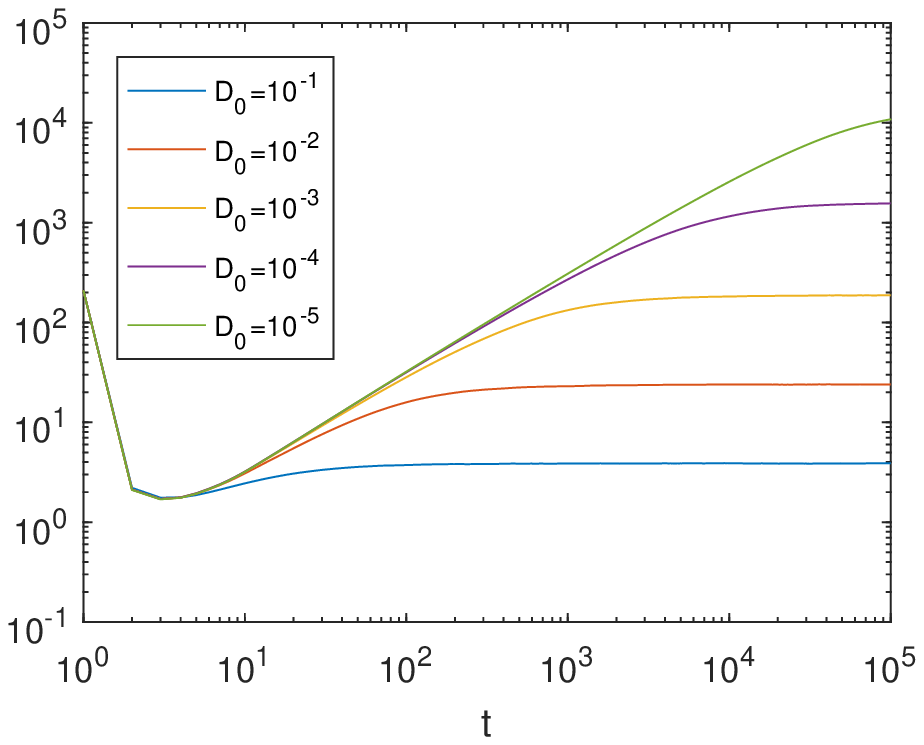}
		\label{fig:abc_result}
	}
	\caption{Time evolution of the $\frac{\mathbb{E}(x_1(t))^2}{2t}$ for different $D_0$ and $\epsilon$. (a) $\epsilon=10$, (b) $\epsilon=1$, (c) $\epsilon=0.1$, (d) $\epsilon=0$.}
	\label{fig:tdABCflowVaryD0VaryEps}
\end{figure}

\begin{figure}[tbph]
	\centering
	\includegraphics[width=0.5\linewidth]{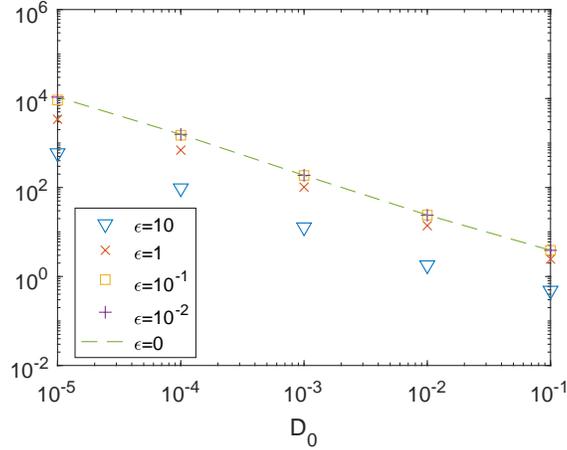}
	\caption{Convection-enhanced diffusion with a maximal enhancement in the time-dependent ABC flow.}
	\label{fig:tdabcflowresult}
\end{figure}

Finally, we investigate the dependence of $D^{E}_{11}$ on the frequency of the time-dependent ABC flow. 
Specifically, we solve the following SDE,
 \begin{equation}\label{TD-ABCflow-Eq-omega}
\begin{cases}
\mathrm{d}x_1=A\sin\big(x_3+ \sin(\Omega t)\big)\mathrm{d}t+C\cos\big(x_2+ \sin(\Omega t)\big)\mathrm{d}t +\sigma \mathrm{d}w_{1,t},\\
\mathrm{d}x_2=B\sin\big(x_1+ \sin(\Omega t)\big)\mathrm{d}t+A\cos\big(x_3+ \sin(\Omega t)\big)\mathrm{d}t +\sigma \mathrm{d}w_{2,t},\\
\mathrm{d}x_3=C\sin\big(x_2+ \sin(\Omega t)\big)\mathrm{d}t+B\cos\big(x_1+  \sin(\Omega t)\big)\mathrm{d}t +\sigma \mathrm{d}w_{3,t},
\end{cases},
\end{equation}
where $A=B=C=1$ and $\Omega$ is the frequency. Here we first choose $\Delta t=2^{-7}$, $N_{mc}=240,000$ and $T=10^5$. Then, we choose different $\Omega$ and compute the corresponding effective diffusivity $D^E_{11}$. 

In Fig.\ref{fig:tdABCflowVaryFreq}, we show the numerical results. We find that when $\Omega$ is near $0.1$ the diffusion enhancement is weak. When $\Omega$ is away from $0.1$, say $\Omega<0.05$ or $\Omega>0.2$, 
we observe the maximal enhancement phenomenon.  
A similar sensitive dependence on the frequency of time-dependent ABC flows was reported in \cite{brummell2001linear}, where the Lyapunov exponent of the deterministic time-dependent ABC flow problem (i.e., $\sigma=0$ in Eq. \eqref{TD-ABCflow-Eq}) was studied as the indicator of the extent of chaos; see Fig.2 and Fig.3 of \cite{brummell2001linear}.  

When $\Omega=0$, the flow of \eqref{TD-ABCflow-Eq-omega} is the same as that for $\epsilon=0$ case in \eqref{TD-ABCflow-Eq}, which will give the maximal enhancement phenomenon. When $\Omega$ is positive, the flow becomes time-dependent and the regions of chaos expand until the extent of chaos (i.e. the Lyapunov exponent) appears to reach a maximum, which is corresponding to $\Omega=0.1$. It seems that the diffusion enhancement is significantly weakened in this range of $\Omega$. When $\Omega$ continues to grow, the islands of the integrability regrow and the chaotic regions have shrunk significantly. We again observe the maximal enhancement phenomenon in this range of $\Omega$. Our numerical results suggest that the level of chaos and the strength of diffusion enhancement seem to compete with each other. More intensive theoretic and numerical studies will be reported in our future work.   

\begin{figure}[tbph]
	\centering
	\includegraphics[width=0.5\linewidth]{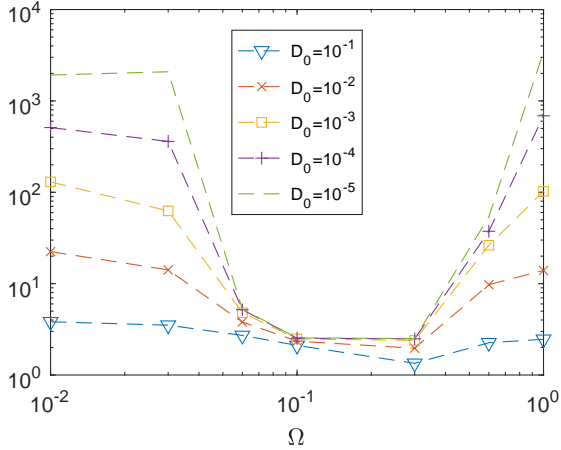}
	\caption{Dependence of $D^{E}_{11}$ on the frequency of the time-dependent ABC flow.}
	\label{fig:tdABCflowVaryFreq}
\end{figure}



\section{Conclusion}
\noindent    
In this paper, we developed a robust stochastic structure-preserving Lagrangian scheme in computing effective diffusivity of passive tracer models and provided a sharp convergence analysis on the proposed numerical scheme. Our convergence analysis is based on a probabilistic approach, which interprets the solution process generated by our numerical scheme as a Markov process. By exploring the ergodicity of the solution process, we gave a sharp and uniform-in-time error estimate for our numerical scheme, which allows us to compute the effective diffusivity over infinite time. Numerical results verify that the proposed method is robust and accurate in computing effective diffusivity of time-dependent chaotic flows. We observed the maximal enhancement phenomenon in time-dependent ABC flows and the sub-maximal enhancement phenomenon in time-dependent Kolmogorov flows, respectively. Moreover, we found that the time dependency in the velocity field improves the chaotic property of ABC flow and Kolmorogov flow though, it does not change the pattern of convection-enhanced diffusion in both flows.

There are two directions we plan to explore in our future work. First, we intend to study the convection-enhanced diffusion phenomenon and provide a sharp convergence analysis for general time-dependent chaotic flows, where the flows have a quasi-periodic property in the time domain. In addition, we shall investigate the convection-enhanced diffusion phenomenon for general spatial-temporal stochastic flows \cite{Yaulandim:1998,Majda:99} and develop convergence analysis for the corresponding numerical methods.

\section*{Acknowledgement}
\noindent
The research of Z. Wang is partially supported by the Hong Kong PhD Fellowship Scheme. The research of J. Xin is partially supported by NSF grants DMS-1924548 and DMS-1952644. The research of Z. Zhang is supported by Hong Kong RGC grants (Projects 17300817 and 17300318), Seed Funding Programme for Basic Research (HKU), and Basic Research Programme (JCYJ20180307151603959) of The Science, Technology and Innovation Commission of Shenzhen Municipality. The computations were performed using research computing facilities offered by Information Technology Services, the University of Hong Kong.

\appendix
\section{Limit of the parameter $\epsilon$ in a time-dependent chaotic flow} \label{sec:Limit-K-Flow} 
\noindent
We shall prove that when $\epsilon$ approaches zero, the effective diffusivity corresponding to the 
time-dependent chaotic flow, e.g. the flow in \eqref{TD-Kflow-Eq} will converge to the one corresponding to the time-independent one, e.g. $\epsilon=0$ in the flow of \eqref{TD-Kflow-Eq}. 
For notational simplicity, let $\textbf{v}=\textbf{v}^\epsilon$ denote the velocity field in \eqref{TD-Kflow-Eq} and $\textbf{v}=\textbf{v}^0$ denote the velocity field when $\epsilon=0$ in $\textbf{v}=\textbf{v}^\epsilon$. Moreover, we denote $\mathcal{L}^\epsilon(\cdot)=\textbf{v}^\epsilon\cdot \nabla_{x}(\cdot)   + D_{0}\Delta_{x}(\cdot) $. Now, the vector corrector field $\boldsymbol\chi^\epsilon$ associated with the velocity field $\textbf{v}^\epsilon$ satisfies the following cell problem,
\begin{equation}\label{eqn:xi-eps}
(\partial_{\tau}+\mathcal{L}^\epsilon)\boldsymbol\chi^\epsilon=-\textbf{v}^\epsilon.
\end{equation}
Let $\boldsymbol\chi^\epsilon_0$ denote the solution of the following equation
\begin{equation}\label{eqn:middle1}
(\partial_{\tau}+\mathcal{L}^\epsilon)\boldsymbol\chi_0^\epsilon=-\textbf{v}^0.
\end{equation}
We aim to prove $\boldsymbol\chi^\epsilon$ converges to $\boldsymbol\chi^\epsilon_0$ as $\epsilon$ approaches zero. At the same time we know the vector corrector field $\boldsymbol\chi_0$ associated with the velocity field $\textbf{v}^0$ satisfies the following cell problem, 
\begin{equation}\label{eqn:xi-0}
\mathcal{L}^0\boldsymbol\chi_0=-\textbf{v}^0,
\end{equation}
where $\mathcal{L}^0(\cdot)=\textbf{v}^0\cdot \nabla_{x} (\cdot) + D_{0}\Delta_{x} (\cdot)$. 
Now we consider, $\boldsymbol\chi_0^0(t,x)=\boldsymbol\chi_0(x)$, which solves,
\begin{equation}\label{eqn:middle2}
(\partial_{\tau}+\mathcal{L}^0)\boldsymbol\chi_0^0=-\textbf{v}^0,
\end{equation}
since $\partial_{\tau}\boldsymbol\chi_0^0=0$. Comparing Eqns.\eqref{eqn:middle1} and \eqref{eqn:middle2} and 
using Prop.\ref{prop:inverseoperator}, we know that $\boldsymbol\chi^\epsilon_0$ converges to $\boldsymbol\chi_0^0$ when $\epsilon$ approaches zero. Finally, comparing Eqns.\eqref{eqn:xi-eps} and \eqref{eqn:middle1}, we know 
$\boldsymbol\chi^\epsilon$ converges to $\boldsymbol\chi_0^\epsilon$ when $\epsilon$ approaches zero. 
Therefore, we prove $\boldsymbol\chi^\epsilon$ converges to $\boldsymbol\chi_0$ when $\epsilon$ approaches zero.

\section*{}
\bibliographystyle{siam}
\bibliography{ZWpaper}

\end{document}